\documentclass[11pt,letterpaper]{article}
\usepackage[margin=1in]{geometry}

\usepackage{amsmath,amssymb,amsthm,mathrsfs,bbm,comment,units,enumitem,xcolor}
\usepackage[colorlinks, linktocpage, allcolors=black,breaklinks]{hyperref}
\usepackage{graphicx,esint,array,pgf,tikz}
\usetikzlibrary{arrows,shapes}
\usetikzlibrary{arrows.meta,bending,automata}
\usetikzlibrary{decorations.pathreplacing}

\usepackage{pdflscape}

\renewcommand{\leq}{\leqslant}
\renewcommand{\geq}{\geqslant}
\newcommand{\concat}{\overset{\frown}{}}

\newtheoremstyle{mythm}
{.5\baselineskip}	
{.5\baselineskip}	
{}		
{}		
{\bf}	
{.\\}		
{ }		
{}		

\theoremstyle{mythm}
\newtheorem{thm}{Theorem}	

\newtheorem{lemma}[thm]{Lemma}
\newtheorem{prop}[thm]{Proposition}

\newtheorem{cor}[thm]{Corollary}
\newtheorem{defn}[thm]{Definition}
\newtheorem{example}[thm]{Example}
\newtheorem*{notn}{Notation}
\newtheorem*{rmk}{Remark}

\newtheoremstyle{myclaim}
{.5\baselineskip}	
{.5\baselineskip}	
{}		
{}		
{\sc}	
{. }		
{ }		
{}		

\theoremstyle{myclaim}

\newcommand{\astfill}{\noindent\xleaders\hbox{$\ast$}\hfill\kern0pt}

\definecolor{SoftRed}{HTML}{DD2222}

\newcommand{\textdefn}[1]{\textcolor{SoftRed}{\textbf{#1}}}

\usepackage[british]{babel}

\newcommand{\dom}{\textrm{dom}}

\usepackage{mathrsfs}

\title{Closed Discrete Selection in the Compact Open Topology}
\author{Christopher Caruvana and Jared Holshouser}
\date{\today}

\begin{document}

\maketitle

\begin{abstract}
	In 2017, Tkachuk isolated the closed discrete selection property while working on problems related to function spaces \cite{Tkachuk2018}. In this paper we will study the closed discrete selection property and the related games and strategies on \(C_k(X)\). Clontz and Holshouser showed previously that the closed discrete selection game on \(C_p(X)\) is equivalent to a modification of the point-open game on \(X\). In this paper we show that the closed discrete selection game on \(C_k(X)\) is equivalent to a modification of the compact-open game on \(X\). We also connect discrete selection properties on \(C_k(X)\) to a variety of other properties on \(X\), \(C_k(X)\), and hyperspaces of \(X\).
\end{abstract}


\section{Introduction}

In 2017, Tkachuk isolated the closed discrete selection property while working on problems related to function spaces \cite{Tkachuk2018}. Tkachuk was able to connect this property for \(C_p(X)\) and \(C_p(X,[0,1])\) to topological properties of \(X\). Tkachuk also studied the game version of the closed discrete selection property and related strategies in that game on \(C_p(X)\) to strategies in the Gruenhage \(W\)-game on \(C_p(X)\) and the point open game on \(X\) \cite{TkachukGame}. Clontz and Holshouser \cite{ClontzHolshouser} strengthened this relationship, showing that strategies for the discrete selection game on \(C_p(X)\) are equivalent to strategies in a non-trivial modification of the point-open game on \(X\). They also related limited information strategies in the closed discrete selection game on \(C_p(X)\) to topological properties of both \(C_p(X)\) and \(X\).

In this paper we will study the closed discrete selection property and the related games and strategies on \(C_k(X)\), the continuous functions from \(X\) to \(\mathbb R\) endowed with the compact-open topology. There is a history of connecting properties of \(C_k(X)\) to properties of both \(X\) and its hyperspaces. We will be referencing and modifying results from Arens \cite{Arens}, Scheepers \cite{Scheepers1997}, and Ko{\v{c}}inac \cite{Kocinac}. We combine these techniques with classical methods, the ideas in Tkachuk's work, and a new approach to game duality currently in development by Clontz \cite{ClontzDuality}.

We have striven to be as general as possible in our methods and have produced robust lists of equivalences for the closed discrete selection principle on \(C_k(X)\), strategies for the corresponding closed discrete game, and limited information strategies for the same game. As in \cite{ClontzHolshouser}, the closed discrete selection game on \(C_k(X)\) is shown to be equivalent to a non-trivial modification of the compact-open game on \(X\).

In this version, we have
\begin{itemize}
    \item
    explicitly stated the convention that only non-trivial open covers are considered.
    \item
    included the class of \(\Lambda\) covers.
	\item
	pointed out errors in Propositions \ref{prop:Pawlikowski1} and \ref{prop:Pawlikowski2} as well as reference \href{https://arxiv.org/abs/2102.00296}{arXiv:2102.00296} where we have recovered the statements for \(k\)-covers.
	\item
	corrected Theorem \ref{thm:FinalBoss}.
\end{itemize}

\section{Definitions}

\begin{defn}
	For a topological space \(X\), we let \(C_p(X)\) denote the set of all continuous functions \(X \to \mathbb R\) endowed with the topology of point-wise convergence.
	We also let \(\mathbf 0\) be the function which identically zero.
\end{defn}
\begin{defn}
	For a topological space \(X\), we let \(C_k(X)\) denote the set of all continuous functions \(X \to \mathbb R\) endowed with the topology of uniform convergence on compact subsets of \(X\).
	We will write
	\[
	    [f;K,\varepsilon] = \left\{ g \in C_k(X) : \sup\{ |f(x)-g(x)| : x \in K \} < \varepsilon \right\}
	\]
	for \(f \in C_k(X)\), \(K \subseteq X\) compact, and \(\varepsilon > 0\).
\end{defn}
\begin{defn}
	For a topological space \(X\), we let \(K(X)\) denote the family of all non-empty compact subsets of \(X\) and \(\mathbb K(X)\) be the set \(K(X)\) endowed with the \textdefn{Vietoris topology} which is the topology generated by sets of the form
	\begin{itemize}
		\item \(\text{below}(U) := \{ K \in K(X) : K \subseteq U \}\) and
		\item \(\text{touch}(U) := \{ K \in K(X) : K \cap U \neq \emptyset \}\)
	\end{itemize}
	for each open \(U \subseteq X\).
	We will write
	\[
	    [U_0;U_1,U_2,\ldots,U_n] = \text{below}(U_0) \cap \bigcap_{j=1}^n \text{touch}(U_j),
	\]
	for \(U_0 , U_1 , U_2 , \ldots , U_n \subseteq X\) open.
	Without loss of generality, we will impose the restrictions that \(U_j \subseteq U_0\) for \(1 \leq j \leq n\) and that \( U_1, U_2, \ldots , U_n\) are pair-wise disjoint.
	For a deeper discussion of this topology, the authors recommend \cite{MichaelVietoris}.
\end{defn}

In this paper, we will be concerned with selection principles and related games.
For classical results, basic tools, and notation, the authors recommend \cite{SakaiScheppers} and \cite{KocinacSelectedResults}.

\begin{defn}
    Given a set \(\mathcal A\) and another set \(\mathcal B\), we define the \textdefn{finite selection principle} \(S_{\text{fin}}(\mathcal A, \mathcal B)\) to be the assertion that, given any sequence \(\{A_n : n \in \omega\} \subseteq \mathcal A\), there exists a sequence \(\{ \mathcal F_n : n \in \omega \}\) so that, for each \(n \in \omega\), \(\mathcal F_n\) is a finite subset of \(A_n\) (denoted as \(\mathcal F_n \in [A_n]^{<\omega}\) hereinafter) and \(\bigcup\{ \mathcal F_n : n \in \omega \} \in \mathcal B\).
\end{defn}
\begin{defn}
    Given a set \(\mathcal A\) and another set \(\mathcal B\), we define the \textdefn{single selection principle} \(S_1(\mathcal A, \mathcal B)\) to be the assertion that, given any sequence \(\{A_n : n \in \omega\} \subseteq \mathcal A\), there exists a sequence \(\{ B_n : n \in \omega \}\) so that, for each \(n \in \omega\), \( B_n \in A_n\) and \(\{ B_n : n \in \omega \} \in \mathcal B\).
\end{defn}

\begin{rmk}
    In general, we impose the condition that open covers be non-trivial; i.e. a collection of open sets \(\mathscr U\) of a space \(X\) is a cover if \(\bigcup \mathscr U = X\) and \(X \not\in \mathscr U\).
\end{rmk}
\begin{defn}
    Let \(X\) be a topological space.
    We say that and open cover \(\mathscr U\) of \(X\) is a \textdefn{\(\lambda\)-cover} if every point is contained in infinitely many members of \(\mathscr U\).
\end{defn}
\begin{defn}
	Let \(X\) be a topological space.
	We say that \(\mathscr U\) is an \textdefn{\(\omega\)-cover} of \(X\) provided that \(\mathscr U\) is an open cover \(X\) with the additional property that, given any finite subset \(F\) of \(X\), there exists some \(U \in \mathscr U\) so that \(F \subseteq U\).
	An infinite \(\omega\)-cover \(\mathscr U\) is said to be a \textdefn{\(\gamma\)-cover} if, for every finite subset \(F \subseteq X\), \(\{ U \in \mathscr U : F \not\subseteq U \}\) is finite.
	If \(\mathscr U = \{U_n : n \in \omega\}\), then \(\mathscr U\) is a \(\gamma\)-cover if and only if every cofinal sequence of the \(U_n\) form an \(\omega\)-cover.
\end{defn}

\begin{defn}
	Let \(X\) be a topological space.
	We say that \(\mathscr U\) is a \textdefn{\(k\)-cover} of \(X\) provided that \(\mathscr U\) is an open cover \(X\) with the additional property that, given any compact subset \(K\) of \(X\), there exists some \(U \in \mathscr U\) so that \(K \subseteq U\).
	An infinite \(k\)-cover \(\mathscr U\) is said to be a \textdefn{\(\gamma_k\)-cover} if, for every compact \(K \subseteq X\), \(\{ U \in \mathscr U : K \not\subseteq U \}\) is finite.
	If \(\mathscr U = \{U_n : n \in \omega\}\), then \(\mathscr U\) is a \(\gamma_k\)-cover if and only if every cofinal sequence of the \(U_n\) form an \(k\)-cover.
\end{defn}

\begin{notn}
	We let
	\begin{itemize}
	    \item
	    \(\mathscr T_X\) denote the set of all non-empty subsets of \(X\).
	    \item
	    \(\Omega_{X,x}\) denote the set of all \(A \subseteq X\) with \(x \in \text{cl}_X(A)\).
	    We also call \(A \in \Omega_{X,x}\) a \textdefn{blade} of \(x\).
	    \item
	    \(\Gamma_{X,x}\) denote the set of all sequences \(\{x_n : n \in \omega\} \subseteq X\) with \(x_n \to x\).
	    \item
	    \(\mathcal D_X\) denote the collection of all dense subsets of \(X\).
	    \item
	    \({CD}_X\) denote the collection of all closed and discrete subsets of \(X\).
		\item
		\(\mathcal O_X\) denote the collection of all open covers of \(X\).
		\item
		\(\Lambda_X\) denote the collection of all \(\lambda\)-covers.
		\item
		\(\Omega_X\) denote the collection of all \(\omega\)-covers of \(X\).
		\item
		\(\Gamma_X\) denote the collection of all \(\gamma\)-covers of \(X\).
		\item
		\(\mathcal K_X\) denote the collection of all \(k\)-covers of \(X\).
		\item
		\(\Gamma_k(X)\) denote the collection of all \(\gamma_k\)-covers of \(X\).
		\item
		For a family of sets \(\mathcal A\), let \(\mathcal O(X, \mathcal A)\) to be all open covers \(\mathscr U\) so that for every \(A \in \mathcal{A}\), there is an open set \(U \in \mathscr U\) which contains \(A\); i.e.
        \[
        \mathcal O(X, \mathcal A) = \{\mathscr{U} \in \mathcal O_X: (\forall A \in \mathcal A)(\exists U \in \mathscr U)[A \subseteq U]\}.
        \]
        \item
		For a family of sets \(\mathcal A\), let \(\Lambda(X, \mathcal A)\) to be all open covers \(\mathscr U\) so that for every \(A \in \mathcal{A}\), there are infinitely many \(U \in \mathscr U\) which contain \(A\); i.e.
        \[
        \Lambda (X, \mathcal A) = \{\mathscr{U} \in \mathcal O_X : (\forall A \in \mathcal A)(\exists^\infty U \in \mathscr U)[A \subseteq U]\}.
        \]
        \item
        For a family of sets \(\mathcal A\), let \(\Gamma(X, \mathcal A)\) to be all infinite open covers \(\mathscr U\) so that for every \(A \in \mathcal{A}\), \(\{ U \in \mathscr U : A \not\subseteq U\}\) is finite.
	\end{itemize}
\end{notn}

\begin{defn}
The following selection principles are known by the following names.
    \begin{itemize}
        \item
        \(S_{\text{fin}}(\Omega_{X,x} , \Omega_{X,x})\) is known as the \textdefn{countable fan tightness} property for \(X\) at \(x \in X\).
        \item
        \(S_1(\Omega_{X,x} , \Omega_{X,x})\) is known as the \textdefn{strong countable fan tightness} property for \(X\) at \(x \in X\).
        \item
        \(S_{\text{fin}}(\mathcal D_X , \Omega_{X,x})\) is known as the \textdefn{countable dense fan tightness} property for \(X\) at \(x \in X\).
        \item
        \(S_1(\mathcal D_X , \Omega_{X,x})\) is known as the \textdefn{strong countable dense fan tightness} property for \(X\) at \(x \in X\).
    \end{itemize}
\end{defn}
\begin{defn}
A space is said to be a \textdefn{Fr{\'{e}}chet-Urysohn} space if, for any \(A \subseteq X\) and \(x \in \text{cl}_X(A)\), there exists a sequence \(\{ x_n : n \in \omega\} \subseteq A\) so that \(x_n \to x\).
A space is said to be a \textdefn{strong Fr{\'{e}}chet-Urysohn} space if, for any sequence \(\{A_n : n \in \omega\}\) and
\[
    x \in \bigcap_{n \in \omega} \text{cl}_X(A),
\]
there exists a sequence \(\{ x_n : n \in \omega\}\) so that, for each \(n \in \omega\), \(x_n \in A_n\) and \(x_n \to x\).
\end{defn}

One easily sees that if \(X\) has countable fan tightness at each of its points, then \(X\) has countable dense fan tightness at each of its points.
Similarly, if \(X\) has countable dense fan tightness at each of its points, then \(X\) is strong Fr{\'{e}}chet-Urysohn.

\(\omega\)-length games with two players are useful for characterizing and calibrating a variety of topological properties. In full generality, these games are played as follows.
\begin{itemize}
    \item The game is played in rounds indexed by the natural numbers. In each round, both players play elements from some set \(\mathcal E\).
    \item The result of a play of the game is a sequence \(A_0,B_0,A_1,B_1,\cdots\).
    \item Which player wins is decided by a set \(\mathcal X \subseteq \mathcal E^\omega\). Player One wins if \((A_0,B_0,A_1,B_1,\cdots) \in \mathcal X\). Otherwise player Two wins. Frequently, part of the win condition will be that player One must play elements of some set \(\mathcal A\) and player Two must play elements of a different set \(\mathcal B\).
\end{itemize}
This can be written in tabular form:
	\[
		\begin{array}{c|ccc}
			\text{I} & A_0 & A_1 & A_2 \cdots\\
			\hline
			\text{II} & B_0 & B_1 & B_2 \cdots
		\end{array}
	\]

A strategy is a function \(\sigma : \mathcal E^{<\omega} \to \mathcal E\). Frequently \(\sigma\) is used to refer to a strategy for player One and \(\tau\) is reserved for player Two. A strategy for player One will produce \(A_n\) from the previous \(n-1\) rounds. A strategy for player Two will produce \(B_n\) from the previous \(n-1\) rounds and \(A_n\).
\begin{itemize}
    \item Player One has winning strategy if there is a strategy \(\sigma\) for One so that no matter what player Two plays in response, One wins the resulting play of the game. We write \(\text{I} \uparrow \mathcal G\).
    \item Player Two has a winning strategy if there is a strategy \(\tau\) for Two so that no matter what player One plays in response, Two wins the resulting play of the game. We write \(\text{II} \uparrow \mathcal G\).
\end{itemize}
If player One has a winning strategy, then player Two does not and vice versa.
It is possible that neither player One nor player Two have a winning strategy.

The strategies just discussed were strategies of perfect information.
It is also possible to have limited information strategies.
\begin{itemize}
    \item A \textdefn{tactic} for One is a strategy which only considers the most recent move from player Two. Formally it is a function \(\sigma:\mathcal E \to \mathcal E\). If One has a winning tactic we write \(\text{I} \underset{\text{tact}}{\uparrow} \mathcal G\).
    \item A \textdefn{Markov strategy} for Two is a strategy which only considers the most recent move of player One and the current turn number. Formally it is a function \(\tau:\mathcal E \times \omega \to \mathcal E\). If Two has a winning Markov strategy we write \(\text{II} \underset{\text{mark}}{\uparrow} \mathcal G\).
    \item A \textdefn{predetermined strategy} for One is a strategy which only consider the current turn number. Formally it is a function \(\sigma:\omega \to \mathcal E\). If One has a winning predetermined strategy we write \(\text{I} \underset{\text{pre}}{\uparrow} \mathcal G\).
    \item All of these strategy types can be defined symmetrically for the other player.
\end{itemize}
\begin{defn}
    Two games \(\mathcal G_1\) and \(\mathcal G_2\) are said to be \textdefn{strategically dual} provided that the following two hold:
    \begin{itemize}
        \item \(\text{I} \uparrow \mathcal G_1 \text{ iff } \text{II} \uparrow \mathcal G_2\)
        \item \(\text{I} \uparrow \mathcal G_2 \text{ iff } \text{II} \uparrow \mathcal G_1\)
    \end{itemize}
    Two games \(\mathcal G_1\) and \(\mathcal G_2\) are said to be \textdefn{Markov dual} provided that the following two hold:
    \begin{itemize}
        \item \(\text{I} \underset{\text{pre}}{\uparrow} \mathcal G_1 \text{ iff } \text{II} \underset{\text{mark}}{\uparrow} \mathcal G_2\)
        \item \(\text{I} \underset{\text{pre}}{\uparrow} \mathcal G_2 \text{ iff } \text{II} \underset{\text{mark}}{\uparrow} \mathcal G_1\)
    \end{itemize}
    Two games \(\mathcal G_1\) and \(\mathcal G_2\) are said to be \textdefn{dual} provided that they are both strategically dual and Markov dual.
\end{defn}

Of particular interest are games derived from selection principles, i.e. selection games.

\begin{defn}
	Given a set \(\mathcal A\) and another set \(\mathcal B\), we define the \textdefn{finite selection game} \(G_{\text{fin}}(\mathcal A, \mathcal B)\) for \(\mathcal A\) and \(\mathcal B\) as follows:
	\[
		\begin{array}{c|ccc}
			\text{I} & A_0 & A_1 & A_2 \cdots\\
			\hline
			\text{II} & \mathcal F_0 & \mathcal F_1 & \mathcal F_2 \cdots
		\end{array}
	\]
	where each \(A_n \in \mathcal A\) and \(\mathcal F_n \in [A_n]^{<\omega}\).
	We declare Two the winner if \(\bigcup\{ \mathcal F_n : n \in \omega \} \in \mathcal B\).
	Otherwise, One wins.
\end{defn}
\begin{defn}
	Similarly, we define the \textdefn{single selection game} \(G_1(\mathcal A, \mathcal B)\) as follows:
	\[
		\begin{array}{c|ccc}
			\text{I} & A_0 & A_1 & A_2 \cdots\\
			\hline
			\text{II} & B_0 & B_1 & B_2 \cdots
		\end{array}
	\]
	where each \(A_n \in \mathcal A\) and \(B_n \in A_n\).
	We declare Two the winner if \(\{ B_n : n \in \omega \} \in \mathcal B\).
	Otherwise, One wins.
\end{defn}

\begin{rmk}
    In general, \(S_1(\mathcal A, \mathcal B)\) is true if and only if \(\text{I} \underset{\text{pre}}{\not\uparrow} G_1(\mathcal{A},\mathcal{B})\), see \cite[Prop. 13]{ClontzHolshouser}.
\end{rmk}

\begin{rmk}
	The game \(G_{\text{fin}}(\mathcal O_X,\mathcal O_X)\) is the well-known Menger game and the game \(G_1(\mathcal O_X, \mathcal O_X)\) is the well-known Rothberger game.
\end{rmk}

\begin{notn}
	For \(A \subseteq X\), let \(\mathscr N(A)\) be all open sets \(U\) so that \(A \subseteq U\). We will write \(\mathscr N_x\) in place of \(\mathscr N(\{x\})\).
	Set \(\mathscr N[X] = \{\mathscr N_x :x \in X\}\), and in general if \(\mathcal A\) is a collection of subsets of \(X\), then \(\mathscr N[\mathcal A] = \{\mathscr N(A) :A \in \mathcal{A}\}\).
\end{notn}
\begin{rmk}
    Oftentimes, when \(\mathscr N[\mathcal A]\) is being used in a game, we will use the identification of \(A\) with \(\mathscr N(A)\) to simplify notation.   
    Particularly, One picks \(A \in \mathcal A\) and Two's response will be an open set \(U\) so that \(A \subseteq U\).
\end{rmk}

\begin{rmk}
	The game \(G_1(\mathscr N[X],\neg \mathcal O_X)\) is the well-known point-open game first appearing in \cite{Galvin1978} and the game \(G_1(\mathscr N[K(X)], \neg \mathcal O_X)\) is the compact-open game.
\end{rmk}

\begin{defn}
    A topological space \(X\) is called \textdefn{discretely selective} if, for any sequence \(\{ U_n : n \in \omega \}\) of non-empty open sets, there exists a closed discrete set \(\{x_n : n \in \omega\} \subseteq X\) so that \(x_n \in U_n\) for each \(n \in \omega\); i.e. \(S_1(\mathscr T_X, {CD}_X)\) holds.
    This notion was first isolated by Tkachuk in \cite{Tkachuk2018}.
\end{defn}
\begin{defn}
	\label{defn:ClosedDiscrete}
	For a topological space \(X\), the \textdefn{closed discrete selection game} on \(X\), is \(G_1(\mathscr T_X, {CD}_X)\).
	Tkachuk studies this game in \cite{TkachukGame}.
\end{defn}
Note that \(X\) is discretely selective if and only if \(\text{I} \underset{\text{pre}}{\not\uparrow} G_1(\mathscr T_X, {CD}_X)\).

\begin{defn}
	\label{defn:ClosureGame}
	For a topological space \(X\) and a point \(x\in X\), the \textdefn{closure game} for \(X\) at \(x\in X\) is \(G_1(\mathscr T_X, \neg \Omega_{X,x})\).
	Tkachuk studies this game in \cite{TkachukGame} as well.
\end{defn}

\begin{defn}
	\label{defn:GruenhageGame}
	For a topological space \(X\) and \(x\in X\), the \textdefn{Gruenhage's \(W\)-game} for \(X\) at \(x\) is \(G_1(\mathscr N_x, \neg \Gamma_{X,x})\).
\end{defn}

\begin{prop}
    For a topological space \(X\), the following are equivalent:
    \begin{enumerate}[label=(\alph*)]
        \item
        \(X\) is strong Fr{\'{e}}chet-Urysohn
        \item
        \(X\) has strong countable fan tightness at each of its points
        \item
        for all \(x\in X\), \(S_1(\Omega_{X,x} , \Omega_{X,x})\)
        \item
        for all \(x\in X\), \(\text{I} \underset{\text{pre}}{\not\uparrow} G_1(\Omega_{X,x} , \Omega_{X,x})\)
    \end{enumerate}
\end{prop}

\section{Results}

\begin{defn}[\cite{ClontzDuality}]
    For a set \(X\), let
    \[
        \mathbf C(X) = \left\{ f : X \to \bigcup X : f(x) \in x \right\}.
    \]
    For a set \(\mathcal A\), we say that \(\mathcal B \subseteq \mathcal A\) is a \textdefn{selection basis} provided that
    \[
        (\forall A \in \mathcal A)(\exists B \in \mathcal B)(B \subseteq A).
    \]
    A set \(\mathcal R\) is said to be a \textdefn{reflection} of a set \(\mathcal A\) if
    \[
        \{ f[\mathcal R] : f \in \mathbf C(\mathcal R) \}
    \]
    is a selection basis for \(\mathcal A\).
\end{defn}
\begin{thm}[{\cite[Corollary 17]{ClontzDuality}}] \label{thm:Clontz}
    If \(\mathcal R\) is a reflection of \(\mathcal A\), then for a set \(\mathcal B\), the two games \(G_1(\mathcal A, \mathcal B)\) and \(G_1(\mathcal R, \neg \mathcal B)\) are dual.
\end{thm}
\begin{prop} \label{prop:Reflection}
    Given a topological space \(X\) and a collection \(\mathcal A\) of subsets of \(X\), \(\mathscr N[\mathcal A]\) is a reflection of \(\mathcal O(X,\mathcal A)\).
\end{prop}
\begin{proof}
    Let \(\mathscr U \in \mathcal O(X,\mathcal A)\).
    Define \(f : \mathcal A \to \mathscr U\) by choosing \(f(A) \in \mathscr U\) so that \(A \subseteq f(A)\).
    Now define \(F : \mathscr N[\mathcal A] \to \bigcup \mathscr N[\mathcal A]\) by
    \[
        F( \mathscr N(A) ) = f(A).
    \]
    Since \(f(A) \in \mathscr U\) for all \(A \in \mathcal A\), we see that
    \[
        F[\mathscr N[\mathcal A]] \subseteq \mathscr U.
    \]
    Hence, \(\mathscr N[\mathcal A]\) is a reflection of \(\mathcal O(X,\mathcal A)\).
\end{proof}
\begin{cor}
    \label{cor:GeneralGalvin}
    Let \(X\) be a topological space, \(\mathcal A\) be a collection of subsets of \(X\), and \(\mathcal B\) be a family of open covers of \(X\).
    Then the games \(G_1(\mathscr N[\mathcal A], \neg\mathcal B)\) and \(G_1(\mathcal O(X,\mathcal A) , \mathcal B)\) are dual.
\end{cor}
\begin{proof}
    Apply Theorem \ref{thm:Clontz} and Proposition \ref{prop:Reflection}.
\end{proof}

This general duality theorem extends the known results that the point-open game and the Rothberger game are dual, see \cite{Galvin1978}, and that the compact-open game and \(G_1(\mathcal K_X, \mathcal O_X)\) are dual, see \cite{Telgarsky1983} and \cite{AurichiDias}.
We also obtain the following as corollaries.

\begin{cor}
    \label{cor:kKODual}
	The games \(G_1(\mathscr N[K(X)],\neg \mathcal K_X)\) and \(G_1(\mathcal K_X, \mathcal K_X)\) are dual.
\end{cor}

\begin{cor}
The games \(G_1(\mathscr N[K(X)], \neg\mathcal O_X)\) and \(G_1(\mathcal K_X , \mathcal O_X)\) are dual.
\end{cor}

\begin{prop}[\cite{AurichiDias}]
	\(\text{I} \uparrow G_1(\mathscr N[K(X)],\neg \mathcal O_X) \Longleftrightarrow \text{II} \uparrow G_{\text{fin}}(\mathcal O_X,\mathcal O_X)\).
\end{prop}

\begin{lemma}
	\(\text{I} \uparrow G_{\text{fin}}(\mathcal O_X,\mathcal O_X) \implies \text{II} \uparrow G_1(\mathscr N[K(X)],\neg \mathcal O_X)\).
\end{lemma}
\begin{proof}
	Suppose \(\text{I} \uparrow G_{\text{fin}}(\mathcal O_X,\mathcal O_X)\) and let \(\sigma\) be a winning strategy.
	It suffices to assume that that One is playing compact sets in \(G_1(\mathscr N[K(X)],\neg \mathcal O_X)\).
	Let \(\mathscr U_0 = \sigma(\emptyset)\) and \(K_0 \subseteq X\) be compact.
	Since \(K_0\) is compact, we let
	\[
		\tau(K_0) = \mathcal F_0 \in [\mathscr U_0]^{<\omega}
	\]
	be so that \(K_0 \subseteq \bigcup \mathcal F_0\).

	Now, for \(n \in \omega\), suppose we have open covers \(\mathscr U_0 , \mathscr U_1 , \ldots , \mathscr U_{n+1}\), compact sets \(K_0 , K_1 , \ldots , K_n\), and \(\mathcal F_j \in [\mathscr U_j]^{<\omega}\) for each \(j \leq n\) so that
	\[
		\langle \mathscr U_0 , \mathcal F_0 , \mathscr U_1 , \mathcal F_1 , \ldots , \mathscr U_n , \mathcal F_n , \mathscr U_{n+1} \rangle \in \sigma
	\]
	and
	\[
		\left\langle K_0 , \bigcup \mathcal F_0 , K_1 , \bigcup \mathcal F_1 , \ldots , K_n , \bigcup \mathcal F_n  \right\rangle \in \tau.
	\]
	Now, for any compact \(K_{n+1} \subseteq X\), we can let \(\mathcal F_{n+1} \in [\mathscr U_{n+1}]^{<\omega}\) be so that \(K_{n+1} \subseteq \bigcup \mathcal F_{n+1}\) and define
	\[
		\tau\left( K_0 , \bigcup \mathcal F_0 , K_1 , \bigcup \mathcal F_1 , \ldots , K_n , \bigcup \mathcal F_n , K_{n+1} \right) = \bigcup \mathcal F_{n+1}
	\]
	This completes the definition of \(\tau\).

	Now, since One wins \(G_{\text{fin}}(\mathcal O_X,\mathcal O_X)\) with \(\sigma\), it must be the case that \(\bigcup\{\mathcal F_n : n \in \omega\}\) fails to be a cover of \(X\).
	That is, if Two plays \(G_1(\mathscr N[K(X)], \neg\mathcal O_X)\) according to \(\tau\), Two wins.
\end{proof}

In \cite{AurichiDias}, it is shown that it is consistent with ZFC that, for a space \(X\), \(S_{\text{fin}}(\mathcal O_X,\mathcal O_X)\) holds but \(S_1(\mathscr N[K(X)], \neg\mathcal O_X)\) fails.
From this example, they ask if it is consistent with ZFC that the games \(G_1(\mathscr N[K(X)], \neg \mathcal O_X)\) and \(G_{\text{fin}}(\mathcal O_X,\mathcal O_X)\) are dual.
This is still open.

\subsection{Covering Properties}

The following is a generalization of a result from Telg{\'{a}}rsky \cite{Telgarsky1975}.

\begin{thm}
    \label{thm:GdeltaCovering}
    Let \(\mathcal A\) be a collection of \(G_\delta\) subsets of \(X\).
    Then the following are equivalent:
    \begin{enumerate}[label=(\alph*)]
        \item \label{GdeltaCoveringGame}
        \(\text{I} \uparrow G_1(\mathscr N[\mathcal A], \neg \mathcal O_X)\)
        \item \label{GdeltaCoveringSequence}
        there is a sequence \(\{A_n : n \in \omega\} \subseteq \mathcal A\) so that \(X = \bigcup_n A_n\)
        \item \label{GdeltaCoveringpreGame}
        \(\text{I} \underset{\text{pre}}{\uparrow} G_1(\mathscr N[\mathcal A], \neg \mathcal O_X)\).
    \end{enumerate}
\end{thm}
\begin{proof}
    \ref{GdeltaCoveringGame} \(\Rightarrow\) \ref{GdeltaCoveringSequence}:
	For each \(A \in \mathcal A\), let \(\mathcal U_A\) be a countable family of descending open sets so that \(A = \bigcap \mathcal U_A\).
	Let \(\sigma\) be a winning strategy for One.
	We will recursively define collections partial plays of \(G_1(\mathscr N[\mathcal A],\neg \mathcal O_X)\) which are being played according to \(\sigma\).
	\begin{itemize}
		\item
		\(T_0 = \emptyset\).
		\item
		Given \(n \in \omega\), define
		\[
			T_{n+1} = \{ w \overset{\frown}{} \langle \sigma(w), U \rangle : w \in T_n \text{ and } U \in \mathcal U_{\sigma(w)} \}.
		\]
	\end{itemize}
	Notice that each \(T_n\) is countable so
	\[
		W := \bigcup_{n\in\omega} \bigcup_{w \in T_n} \sigma(w)
	\]
	is a countable union of sets from \(\mathcal A\).

	To finish, we show that \(X = W\).
	Toward this end, suppose that, by way of contradiction, there exists \(x \in X \setminus W\).
	Then
	\begin{itemize}
		\item
		\(x \not\in A_0 := \sigma(\emptyset)\) and we can find \(U_0 \in \mathcal U_{A_0}\) so that \(x \not\in U_0\).
		\item
		Suppose we have
		\[
			w = \langle A_0, U_0, A_1, U_1, \ldots, A_n, U_n \rangle \in T_{n+1}
		\]
		defined.
		Then \(x \not\in A_{n+1} := \sigma(w)\) whence we can find \(U_{n+1} \in \mathcal U_{A_{n+1}}\) so that \(x \not\in U_{n+1}\).
	\end{itemize}
	Continuing in this way, we produce a run of the game which contradicts that \(\sigma\) is winning for One.
	Therefore, \(X = W\).

	\ref{GdeltaCoveringSequence} \(\Rightarrow\) \ref{GdeltaCoveringpreGame}:
	Suppose \(X = \bigcup_n A_n\).
	Define \(\sigma\) so that at round \(n\), \(\sigma\) will play \(\mathscr N(A_n)\).
	This is obviously a pre-determined winning strategy.

    \ref{GdeltaCoveringpreGame} \(\Rightarrow\) \ref{GdeltaCoveringGame}:
    Obvious.
\end{proof}

\begin{cor}
	\label{cor:PointOpenCountable}
	Suppose \(X\) is so that each singleton set is a \(G_\delta\).
	Then \(\text{I} \uparrow G_1(\mathscr N[X], \neg\mathcal O_X)\) (the point-open game) if and only if \(X\) is countable.
\end{cor}

\begin{cor}
	\label{cor:CompactOpenSigma}
	Suppose \(X\) is so that each compact set is a \(G_\delta\).
	Then \(\text{I} \uparrow G_1(\mathscr N[K(X)], \neg \mathcal O_X)\) (the compact-open game) if and only if \(X\) is \(\sigma\)-compact.
\end{cor}

\begin{thm}
    \label{thm:GdeltaSuperCovering}
    Let \(\mathcal A\) be a collection of \(G_\delta\) subsets of \(X\).
    Then the following are equivalent:
    \begin{enumerate}[label=(\alph*)]
        \item \label{GdeltaSuperCoveringGame}
        \(\text{I} \uparrow G_1(\mathscr N[\mathcal A], \neg \mathcal O(X, \mathcal A))\)
        \item \label{GdeltaSuperCoveringSequence}
        there is a sequence \(\{A_n : n \in \omega\} \subseteq \mathcal A\) with the property that for each \(A \in \mathcal A\), there exists an \(n\) so that \(A \subseteq A_n\) and \(X = \bigcup_n A_n\)
        \item \label{GdeltaSuperCoveringpreGame}
        \(\text{I} \underset{\text{pre}}{\uparrow} G_1(\mathscr N[\mathcal A], \neg \mathcal O(X, \mathcal A))\)
    \end{enumerate}
\end{thm}
\begin{proof}
    \ref{GdeltaSuperCoveringGame} \(\Rightarrow\) \ref{GdeltaSuperCoveringSequence}:
	For each \(A \in \mathcal A\), let \(\mathcal U_A\) be a countable family of descending open sets so that \(A = \bigcap \mathcal U_A\).
	Let \(\sigma\) be a winning strategy for One and, without loss of generality, suppose \(\sigma\) is producing elements of \(\mathcal A\).
	We will recursively define collections partial plays of \(G_1(\mathscr N[\mathcal A], \neg\mathcal O(X, \mathcal A))\) which are being played according to \(\sigma\).
	\begin{itemize}
		\item
		\(T_0 = \emptyset\).
		\item
		Given \(n \in \omega\), define
		\[
			T_{n+1} = \{ w \overset{\frown}{} \langle \sigma(w), U \rangle : w \in T_n \text{ and } U \in \mathcal U_{\sigma(w)} \}.
		\]
	\end{itemize}
	Notice that each \(T_n\) is countable so
	\[
		\mathscr F := \bigcup_{n\in\omega} \{ \sigma(w) : w \in T_n \}
	\]
	is a countable family of sets from \(\mathcal A\).

	Suppose, toward a contradiction, that \(A \in \mathcal A\) is so that
	\[
		(\forall F \in \mathscr F)(A \not\subseteq F).
	\]
	Then
	\begin{itemize}
		\item
		for \(A_0 = \sigma(\emptyset)\), we can find \(x_0 \in A \setminus A_0\) and \(U_0 \in \mathcal U_{K_0}\) so that \(x_0 \not\in U_0\).
		\item
		Suppose we have
		\[
			w = \langle A_0, U_0, A_1, U_1, \ldots, A_n, U_n \rangle \in T_{n+1}
		\]
		defined.
		Then, for \(A_{n+1} = \sigma(w)\), we can find \(x_{n+1} \in A \setminus A_{n+1}\) and \(U_{n+1} \in \mathcal U_{K_{n+1}}\) so that \(x_{n+1} \not\in U_{n+1}\).
	\end{itemize}
	Continuing in this way, we obtain a full play
	\[
		A_0 , U_0 , A_1 , U_1 , \ldots
	\]
	of \(G_1(\mathscr N[K(X)], \neg\mathcal K_X)\) according to \(\sigma\).
	That is, there must be some \(n \in \omega\) so that \(A \subseteq U_n\), contradicting \(x_n \not\in U_n\).

	\ref{GdeltaSuperCoveringSequence} \(\Rightarrow\) \ref{GdeltaSuperCoveringpreGame}:
	Suppose \(\{A_n : n \in \omega\} \subseteq \mathcal A\) has the property that for each \(A \in \mathcal A\), there exists an \(n\) so that \(A \subseteq A_n\) and \(X = \bigcup_n A_n\).
	Define \(\sigma\) so that at round \(n\), \(\sigma\) will play \(\mathscr N(A_n)\).
	This is obviously a pre-determined winning strategy.

	\ref{GdeltaSuperCoveringpreGame} \(\Rightarrow\) \ref{GdeltaSuperCoveringGame}:
	Obvious.
\end{proof}

\begin{cor}
	\label{cor:kCompactOpenHemi}
	Suppose \(X\) is so that each compact set is a \(G_\delta\).
	Then \(\text{I} \uparrow G_1(\mathscr N[K(X)], \neg\mathcal K_X)\) if and only if \(X\) is hemicompact.
\end{cor}

\begin{example}
    Corollaries \ref{cor:CompactOpenSigma} and \ref{cor:kCompactOpenHemi} demonstrate that \(G_1(\mathscr N[K(X)],\neg \mathcal O_X)\) and \(G_1(\mathscr N[K(X)],\neg \mathcal K_X)\) are different games since \(\mathbb Q\) is \(\sigma\)-compact but not hemicompact.
    Explicitly, \(\text{I} \uparrow G_1(\mathscr N[K(\mathbb Q)] ,\neg \mathcal O_{\mathbb Q})\) but \(\text{I} \not\uparrow G_1(\mathscr N[K(\mathbb Q)],\neg \mathcal K_{\mathbb Q})\).
\end{example}

The following is a generalization of V.416 from \cite[p. 460]{TkachukFE}.

\begin{lemma}
	\label{lem:CompactOpenCofinal}
	Let \(\mathcal A\) be a family of sets closed under finite unions.
	Then \(\text{I} \uparrow G_1(\mathscr N[\mathcal A], \neg\mathcal O(X,\mathcal A))\) if and only if \(\text{I} \uparrow G_1(\mathscr N[\mathcal A],\neg\Gamma(X,\mathcal A))\).
\end{lemma}
\begin{proof}
	Let \(s\) be a winning strategy for One in \(G_1(\mathscr N[\mathcal A], \neg\mathcal O(X,\mathcal A))\).
	We call a finite tuple \(\langle U_0 , U_1 , \ldots , U_n \rangle\) of open subsets of \(X\) \emph{adequate} if there exist \(A_j \in \mathcal A\), \(0 \leq j \leq n\), so that
	\[
		\langle A_0 , U_0 , A_1 , U_1 , \ldots , A_n \rangle \in s
	\]
	and \(A_n \subseteq U_n\).
	Also, we will call a sequence \(\{ U_n : n \in \omega \}\) of open sets \emph{adequate} if, for each \(n \in \omega\), \(\langle U_0 , U_1 , \ldots , U_n \rangle\) is adequate.

	Now, suppose \(\{ U_n : n \in \omega \}\) is adequate and let \(n \in \omega\).
	Then, let \(A_0 , A_1 , \ldots , A_n \in \mathcal A\) so that
	\[
		\langle A_0 , U_0 , A_1 , U_1 , \ldots , A_n \rangle \in s
	\]
	and \(A_n \subseteq U_n\).
	Notice that \(A_0 = s(\emptyset)\).
	Consider
	\[
		A_{n+1} := s(A_0 , U_0 , A_1 , U_1 , \ldots , A_n , U_n)
	\]
	and let \(A'_0 , A'_1 , \ldots , A'_{n+1} \in \mathcal A\) be so that
	\[
		\langle A'_0 , U_0 , A'_1 , U_1 , \ldots , A'_{n+1} \rangle \in s
	\]
	and \(A'_{n+1} \subseteq U_{n+1}\).
	Again, observe that \(A'_0 = s(\emptyset) = A_0\).

	To see that \(A_j = A'_j\) for all \(j \leq n+1\), fix \(\ell \geq 0\) and suppose that we've shown \(A_j = A'_j\) for all \(j \leq \ell\).
	Then, since we have been playing according to \(s\), we see that
	\[
		\begin{array}{rcl}
			A_{\ell+1}
			&=& s(A_0 , U_0 , A_1 , U_1 , \ldots , A_\ell , U_\ell)\\
			&=& s(A'_0 , U_0 , A'_1 , U_1 , \ldots , A'_\ell , U_\ell)\\
			&=& A'_{\ell+1}.
		\end{array}
	\]
	It follows that \(\{U_n:n\in\omega\}\) arises as the sequence of Two's plays in a full run of the game \(G_1(\mathscr N[\mathcal A], \neg\mathcal O(X,\mathcal A))\) according to \(s\).
	In particular, if \(\{U_n : n \in \omega\}\) is an adequate family, \(\{U_n : n \in \omega\} \in \mathcal A\).

	For any adequate sequence \(\langle U_0 , U_1 , \ldots , U_n\rangle\), let \(w(U_0 , U_1 , \ldots , U_n) = \langle A_0 , A_1 , \ldots , A_n \rangle\) be so that
	\[
		\langle A_0 , U_0 , \ldots , A_n , U_n \rangle
	\]
	is a play according to \(s\) and define
	\[
		\gamma(U_0 , U_1 , \ldots , U_n) = s( A_0 , U_0 , \ldots , A_n , U_n ) \cup \bigcup_{j=0}^n A_j.
	\]
	Observe that \(\gamma(U_0 , U_1 , \ldots , U_n) \in \mathcal A\) as it is a finite union of sets from \(\mathcal A\).

	Now we will define a new strategy \(\sigma\) and we start with \(\sigma(\emptyset) = A_0 = s(\emptyset)\).
	Suppose that we have defined
	\[
		\langle A_0 , U_0 , A_1 , U_1 , \ldots , A_{n-1} , U_{n-1} , A_n \rangle \in \sigma
	\]
	for \(n \geq 0\) so that, for a fixed open set \(U_n\) with \(A_n \subseteq U_n\), we have
	\begin{enumerate}[label=(\roman*)]
		\item \label{PieceTogetherA}
		for any \(0 \leq j_0 < j_1 < \cdots < j_k \leq n\), \(\langle U_{j_0} , U_{j_1} , \ldots , U_{j_k} \rangle\) is an adequate sequence and
		\item \label{PieceTogetherB}
		for any \(0 \leq j_0 < j_1 < \cdots < j_k \leq \ell < n\), \(\gamma( U_{j_0} , U_{j_1} , \ldots , U_{j_k} ) \subseteq A_{\ell+1}\).
	\end{enumerate}
	Define
	\[
		A_{n+1} = \bigcup \{ \gamma(U_{j_0} , U_{j_1} , \ldots , U_{j_k}) : 0 \leq j_0 < j_1 < \cdots < j_k \leq n \}
	\]
	which is a set in \(\mathcal A\), let
	\[
		\langle A_0 , U_0 , A_1 , U_1 , \ldots , A_{n-1} , U_{n-1} , A_n , U_n , A_{n+1} \rangle \in \sigma,
	\]
	and fix an open set \(U_{n+1}\) with \(A_{n+1} \subseteq U_{n+1}\).

	To address \ref{PieceTogetherA}, let \(0 \leq j_0 < j_1 < \cdots < j_{k-1} < j_k = n + 1\).
	Notice that \(\langle U_{j_0} , U_{j_1} , \ldots , U_{j_{k-1}} \rangle\) is adequate by the inductive hypothesis so let
	\[
		\langle B_{j_0} , B_{j_1} , \ldots , B_{j_{k-1}} \rangle = w(U_{j_0} , U_{j_1} , \ldots , U_{j_{k-1}})
	\]
	and
	\[
		B_{j_k} = s(B_{j_0} , U_{j_0} , B_{j_1} , U_{j_1} , \ldots , B_{j_{k-1}} , U_{j_{k-1}}).
	\]
	It follows that
	\[
		B_{j_k} \subseteq \gamma(U_{j_0} , U_{j_1} , \ldots , U_{j_{k-1}}) \subseteq A_{n+1} \subseteq U_{n+1}.
	\]
	Hence,
	\[
		B_{j_0} , U_{j_0} , B_{j_1} , U_{j_1} , \ldots , B_{j_{k-1}} , U_{j_{k-1}} , B_{j_k} , U_{j_k}
	\]
	is a play according to \(s\); i.e., \(\langle U_{j_0} , U_{j_1} , \ldots , U_{j_k} \rangle\) is adequate.

	The fact that \ref{PieceTogetherB} holds for \(0 \leq j_0 < j_1 < \cdots < j_k \leq \ell < n + 1\) follows immediately from the definition of \(A_{n+1}\).
	So this finishes the construction of \(\sigma\).

	Suppose \(A_0 , U_0 , A_1 , U_1 , \ldots\) is a full run of the game \(G_1(\mathscr N[\mathcal A], \neg\mathcal O(X, \mathcal A))\) played according to \(\sigma\) and suppose, by way of contradiction, that there is \(A \in \mathcal A\) so that
	\[
		(\forall n \in \omega)(\exists m \geq n)(A \not\subseteq U_m).
	\]
	Then we can build a co-final sequence \(j_0 < j_1 < \cdots\) so that \(A \not\subseteq U_{j_n}\) for each \(n \in \omega\).

	By construction of \(\sigma\), \(\langle U_{j_0} , U_{j_1} , \ldots , U_{j_n} \rangle\) is adequate for each \(n \in \omega\) which means that \(\{ U_{j_n} : n \in \omega \}\) is adequate.
	By above, \(\{ U_{j_n} : n \in \omega\} \in \mathcal O(X,\mathcal A)\).
	In particular, there must be some \(n \in \omega\) so that \(A \subseteq U_{j_n}\), a contradiction.
	Therefore, \(\{U_n :n \in \omega\} \in \Gamma(X,\mathcal A)\), and \(\sigma\) is a winning strategy for One in \(G_1(\mathscr N[\mathcal A],\neg\Gamma(X,\mathcal A))\).
\end{proof}

\begin{cor}
    \label{cor:CompactOpenCofinal}
    We have the following strengthenings of strategies.
    \begin{enumerate}[label=(\roman*)]
        \item \label{Immediate}
        \(\text{I} \uparrow G_1(\mathscr N[K(X)], \neg \mathcal K_X)\) if and only if \(\text{I} \uparrow G_1(\mathscr N[K(X)] ,\neg \Gamma_k(X))\).
        \item \label{ImmediateTwo}
        \(\text{II} \uparrow G_1(\mathcal K_X, \mathcal K_X)\) if and only if \(\text{II} \uparrow G_1(\mathcal K_X, \Gamma_k(X))\).
        \item \label{ImmediatePre}
        \(\text{I} \underset{\text{pre}}{\uparrow} G_1(\mathscr N[K(X)],\neg \mathcal K_X)\) if and only if \(\text{I} \underset{\text{pre}}{\uparrow} G_1(\mathscr N[K(X)] , \neg\Gamma_k(X))\).
        \item \label{ImmediateMark}
        \(\text{II} \underset{\text{mark}}{\uparrow} G_1(\mathcal K_X, \mathcal K_X)\) if and only if \(\text{II} \underset{\text{mark}}{\uparrow} G_1(\mathcal K_X, \Gamma_k(X))\).
    \end{enumerate}
\end{cor}
\begin{proof}
    \ref{Immediate} This is immediate from Lemma \ref{lem:CompactOpenCofinal}.

    \ref{ImmediateTwo} Apply \ref{Immediate} and Corollaries \ref{cor:GeneralGalvin} and \ref{cor:kKODual}.

    \ref{ImmediatePre} Note that if \(s\) is a pre-determined winning strategy for One in \(G_1(\mathscr N[K(X)], \neg\mathcal K_X)\), then
    \[
        \sigma(n) = \bigcup_{i \leq n}s(i)
    \]
    is a pre-determined winning strategy for One in \(G_1(\mathscr N[K(X)] ,\neg \Gamma_k(X))\).
    To see this, note that any cofinal play of the game according to \(\sigma\) can be unraveled into a play of the game according to \(s\) by having player One play each \(s(i)\) and having player Two play their response to \(\sigma(n)\) repeatedly against the \(s(i)\).

    \ref{ImmediateMark}
    Apply \ref{ImmediatePre} and Corollary \ref{cor:GeneralGalvin}.
\end{proof}

\begin{lemma}
	\(\text{I} \uparrow G_{\text{fin}}(\mathcal O_X,\mathcal O_X) \implies \text{I} \uparrow G_1(\mathcal K_X, \mathcal K_X)\).
\end{lemma}
\begin{proof}
	Let \(s\) be a winning strategy for One in \(G_{\text{fin}}(\mathcal O_X,\mathcal O_X)\).
	For any open cover \(\mathscr U\), let
	\[
		\mathscr U^+ = \left\{ \bigcup \mathcal F : \mathcal F \in [\mathscr U]^{<\omega} \right\}.
	\]

	Now, we define \(\sigma\):
	Let \(\sigma(\emptyset) = \mathscr U_0 := s(\emptyset)^+\).
	Since \(s(\emptyset)\) is an open cover of \(X\), \(\mathscr U_0\) is a \(k\)-cover of \(X\).

	For \(n \in \omega\), suppose we have
	\[
		\langle \mathscr V_0 , \mathcal F_0 , \ldots, \mathscr V_{n-1} , \mathcal F_{n-1} , \mathscr V_n \rangle \in s
	\]
	and
	\[
		\langle \mathscr U_0 , U_0 , \ldots, \mathscr U_{n-1} , U_{n-1} , \mathscr U_n \rangle \in \sigma
	\]
	where \(\mathscr U_j = \mathscr V_j^+\) for each \(j \leq n\) and \(U_j = \bigcup \mathcal F_j\) for each \(j < n\).
	For any choice \(U_n \in \mathscr U_n\), there exists \(\mathcal F_n \in [\mathscr V_n]^{<\omega}\) so that \(U_n = \bigcup \mathcal F_n\).
	Then we obtain some open cover \(\mathscr V_{n+1}\) with
	\[
		\langle \mathscr V_0 , \mathcal F_0 , \ldots, \mathscr V_{n-1} , \mathcal F_{n-1} , \mathscr V_n , \mathcal F_n , \mathscr V_{n+1} \rangle \in s
	\]
	and we let \(\mathscr U_{n+1} = \mathscr V_{n+1}^+\) and
	\[
		\langle \mathscr U_0 , U_0 , \ldots, \mathscr U_{n-1} , U_{n-1} , \mathscr U_n , U_{n} , \mathscr U_{n+1} \rangle \in \sigma.
	\]
	This finishes the definition of \(\sigma\).

	Since \(s\) is winning for One in \(G_{\text{fin}}(\mathcal O_X,\mathcal O_X)\), \(\bigcup\{\mathcal F_n : n \in \omega\}\) is not a cover of \(X\).
	It follows that \(\{U_n : n \in \omega\}\) is not a cover of \(X\) and, therefore, \(\sigma\) is a winning strategy for One in \(G_1(\mathcal K_X, \mathcal K_X)\).
\end{proof}

\subsection{Topological Characterizations and Equivalent Games}

\begin{thm}
    Let \(\mathcal A\) be a collection of closed subsets of \(X\) so that each point of \(X\) can be separated from each \(A\in \mathcal A\); i.e., for each \(A\in \mathcal A\) and \(x\in X\setminus A\), there exists an open set \(U\) so that \(A \subseteq U\) and \(x \not\in U\).
    Then \(\text{I} \underset{\text{pre}}{\uparrow} G_1(\mathscr N[\mathcal A], \neg\mathcal O_X)\) if and only if there is a sequence \(\{A_n : n \in \omega\} \subseteq \mathcal A\) so that \(X = \bigcup_n A_n\).
\end{thm}
\begin{proof}
	(\(\Rightarrow\))
	Let \(\{A_n : n \in \omega\} \subseteq \mathcal A\) so that \(X = \bigcup \{ A_n : n \in \omega \}\).
	Then let One play \(A_n\) in the \(n^{\text{th}}\) inning.

	(\(\Leftarrow\))
	Suppose One has a predetermined winning strategy \(\sigma\) and let \(A_n = \sigma(n)\) for each \(n \in \omega\).
	By way of contradiction, suppose there is some \(x \in X \setminus \bigcup\{A_n:n\in\omega\}\).
	For each \(n \in \omega\), we can find an open set \(U_n\) so that \(A_n \subseteq U_n\) and \(x \not\in U_n\).
	Now, we have a contradiction to the fact that \(\sigma\) is a winning predetermined strategy since \(x \not\in \bigcup\{U_n:n\in\omega\}\).
\end{proof}

\begin{cor}
	A \(T_1\) space \(X\) is countable if and only if \(\text{I} \underset{\text{pre}}{\uparrow} G_1(\mathscr N[X],\neg \mathcal O_X)\).
\end{cor}

\begin{cor}
	A Hausdorff space \(X\) is \(\sigma\)-compact if and only if \(\text{I} \underset{\text{pre}}{\uparrow} G_1(\mathscr N[K(X)], \neg\mathcal O_X)\).
\end{cor}

\begin{thm}
    \label{thm:StrongPreCovering}
    Let \(\mathcal A\) be a collection of closed subsets of \(X\) so that each point of \(X\) can be separated from each \(A\in \mathcal A\); i.e., for each \(A\in \mathcal A\) and \(x\in X\setminus A\), there exists an open set \(U\) so that \(A \subseteq U\) and \(x \not\in U\).
    Then \(\text{I} \underset{\text{pre}}{\uparrow} G_1(\mathscr N[\mathcal A],\neg \mathcal O(X,\mathcal A))\) if and only if there is a sequence \(\{A_n : n \in \omega\} \subseteq \mathcal A\) so that \(X = \bigcup_n A_n\) and for each \(A\in \mathcal A\), there exists an \(n\) so that \(A \subseteq A_n\).
\end{thm}
\begin{proof}
    (\(\Rightarrow\))
    Let \(\{A_n : n \in \omega\} \subseteq \mathcal A\) be so that \(X = \bigcup \{ A_n : n \in \omega \}\) and, for every \(A \in \mathcal A\), there exists \(n \in \omega\) so that \(A \subseteq A_n\).
	Then let One play \(A_n\) in the \(n^{\text{th}}\) inning.

	(\(\Leftarrow\))
	Suppose One has a predetermined winning strategy \(\sigma\) and let \(A_n = \sigma(n)\) for each \(n \in \omega\).
	By way of contradiction, suppose there is some \(A \in \mathcal A\) so that, for each \(n \in \omega\), \(A \not\subseteq A_n\).
	Then, let \(x_n \in A \setminus A_n\) for each \(n \in \omega\).
	For each \(n \in \omega\), we can find an open set \(U_n\) so that \(A_n \subseteq U_n\) and \(x_n \not\in U_n\).
	But then \(A_0 , U_0 , A_1 , U_1 , \ldots\) is a full run of \(G_1(\mathscr N[K(X)],\neg \mathcal O(X, \mathcal A))\) played according to \(\sigma\) so there must be some \(n \in \omega\) for which \(x_n \in A \subseteq U_n\), a contradiction.
\end{proof}

In Theorem \ref{thm:Hemicompact}, the implications
\begin{center}
    \ref{hemi} \(\Rightarrow\) \ref{metri} \(\Rightarrow\) \ref{firstly} \(\Rightarrow\) \ref{hemi}
\end{center}
are adapted from \cite{Arens} and we have included the proof for the sake of clarity.

\begin{thm}
    \label{thm:Hemicompact}
	For a Tychonoff space \(X\), the following are equivalent.
	\begin{enumerate}[label=(\alph*)]
	    \item \label{hemi}
	    \(X\) is hemicompact
	    \item \label{VietorisHemi}
	    \(\mathbb K(X)\) is hemicompact
	    \item \label{VietorisSigmaCompact}
	    \(\mathbb K(X)\) is \(\sigma\)-compact
	    \item \label{metri}
	    \(C_k(X)\) is metrizable
	    \item \label{firstly}
	    \(C_k(X)\) is first-countable
	    \item \label{prekKO}
	    \(\text{I} \underset{\text{pre}}{\uparrow} G_1(\mathscr N[K(X)], \neg\mathcal K_X)\), a variant of the compact-open game
	    \item \label{markKRoth}
	    \(\text{II} \underset{\text{mark}}{\uparrow} G_1(\mathcal K_X, \mathcal K_X)\), a variant of the Rothberger game
	    \item \label{preGru}
	    \(\text{I} \underset{\text{pre}}{\uparrow} G_1(\mathscr N_{\mathbf 0}, \neg \Gamma_{C_k(X),\mathbf 0})\), Gruenhage's W-game
	    \item \label{preCL}
	    \(\text{I} \underset{\text{pre}}{\uparrow} G_1(\mathscr N_{\mathbf 0}, \neg \Omega_{C_k(X),\mathbf 0})\), the closure game
	    \item \label{preCD}
	    \(\text{I} \underset{\text{pre}}{\uparrow} G_1(\mathscr T_{C_k(x)}, CD_{C_k(X)})\), the closed discrete selection game
        \item \label{markCkCFT}
        \(\text{II} \underset{\text{mark}}{\uparrow} G_1(\Omega_{C_k(X), \mathbf 0}, \Omega_{C_k(X), \mathbf 0})\), the countable fan tightness game
        \item \label{markCkCDFT}
        \(\text{II} \underset{\text{mark}}{\uparrow} G_1(\mathcal D_{C_k(X)}, \Omega_{C_k(X), \mathbf 0})\), the countable dense fan tightness game
	 \end{enumerate}
\end{thm}
\begin{proof}
    \ref{hemi} \(\Rightarrow\) \ref{VietorisHemi}:
    Let \(\{K_n:n\in\omega\}\) be a collection of compact sets so that, for every compact \(K \subseteq X\), there exists \(n\in\omega\) so that \(K \subseteq K_n\).
    Notice that
    \[
        \mathscr F_n := \{ K \in \mathbb K(X) : K \subseteq K_n \}
    \]
    is compact for each \(n\in\omega\) and that
    \[
        \mathbb K(X) = \bigcup_{n\in\omega} \mathscr F_n.
    \]
    To see that \(\mathbb K(X)\) is actually hemicompact, let \(\mathscr F \subseteq \mathbb K(X)\) be compact.
    Then \(\bigcup \mathscr F \subseteq X\) is compact which means there exists \(n\in\omega\) so that \(\bigcup \mathscr F \subseteq K_n\).
    It follows that \(\mathscr F \subseteq \mathscr F_n\), establishing the hemicompactness of \(\mathbb K(X)\).

    \ref{VietorisHemi} \(\Rightarrow\) \ref{VietorisSigmaCompact}:
    Obvious.

    \ref{VietorisSigmaCompact} \(\Rightarrow\) \ref{hemi}:
    Let \(\{\mathscr F_n:n\in\omega\}\) be a collection of compact subsets of \(\mathbb K(X)\) so that
    \[
        \mathbb K(X) = \bigcup_{n\in\omega}\mathscr F_n.
    \]
    Then notice that
    \[
        K_n := \bigcup \mathscr F_n
    \]
    is compact in \(X\).
    Let \(K \subseteq X\) be compact.
    Then there exists \(n \in \omega\) so that \(K \in \mathscr F_n\) which provides that \(K \subseteq K_n\).

    \ref{hemi} \(\Rightarrow\) \ref{metri}:
    Let \(\{ K_n : n \in \omega \}\) be an ascending sequence of compact subsets of \(X\) so that, for each compact \(K \subseteq X\), there exists \(n \in \omega\) so that \(K \subseteq K_n\).
    For each \(n \in \omega\) and \(f , g \in C_k(X)\), define
    \[
        \| f - g \|_n = \min\{ \sup\{ |f(x) - g(x) | : x \in K_n \} , 1 \}
    \]
    and define \(d : C_k(X)^2 \to [0,1]\) by
    \[
        d(f,g) = \sum_{n = 0}^\infty \frac{\|f - g\|_n}{2^{n+1}}.
    \]
    It is straight forward to verify that \(d\) is indeed a metric.

    To see that this metric is compatible with the topology on \(C_k(X)\), first consider a basic neighborhood \([f;K,\varepsilon]\) about \(f\).
    Then choose \(n \in \omega\) so large that \(2^{-(n+1)} < \varepsilon\) and \(K \subseteq K_n\).
    Let \(g \in B_d(f,\delta)\) where
    \[
        \delta = \frac{\varepsilon - 2^{-(n+1)}}{n+1} > 0.
    \]
    Behold that
    \[
        \frac{\|f-g\|_n}{2^{n+1}}
        \leq d(f,g)
        < \frac{\varepsilon - 2^{-(n+1)}}{n+1}
    \]
    which provides
    \[
        (n+1) \cdot \frac{\|f-g\|_n}{2^{n+1}} < \varepsilon - 2^{-(n+1)}
        \implies
        \sum_{j=0}^n \frac{\|f-g\|_n}{2^{j+1}}
        < \varepsilon - 2^{-(n+1)}.
    \]
    Moreover,
    \begin{eqnarray*}
        \|f-g\|_n
        &=& \sum_{j=0}^n \frac{\|f-g\|_n}{2^{j+1}} + \sum_{j = n+1}^\infty \frac{\|f-g\|_n}{2^{j+1}}\\
        &<& \varepsilon - 2^{-(n+1)} + \sum_{j = n+1}^\infty \frac{\|f-g\|_n}{2^{j+1}}\\
        &\leq& \varepsilon - 2^{-(n+1)} + \sum_{j = n+1}^\infty \frac{1}{2^{j+1}}\\
        &=& \varepsilon.
    \end{eqnarray*}
    In particular \(|f(x) - g(x)| < \varepsilon\) for every \(x \in K_n\) and, since \(K \subseteq K_n\), we see that this necessitates
    \[
        B_d(f;\delta) \subseteq [f;K,\varepsilon].
    \]

    Now, consider a basic \(d\)-ball \(B_d(f,\varepsilon)\) centered at \(f\) and choose \(n\in\omega\) large enough so that
    \[
        \sum_{j = n+1}^\infty \frac{1}{2^{j+1}} < \frac{\varepsilon}{2}.
    \]
    Let \(g \in [f;K_n,\varepsilon/2]\) and notice that, since the \(\{K_n :n \in\omega\}\) were chosen to be ascending, \(\|f-g\|_n \leq \|f-g\|_{n+1}\) for all \(n \in \omega\).
    It follows that
    \begin{eqnarray*}
        d(f,g)
        &=& \sum_{j = 0}^\infty \frac{\|f-g\|_j}{2^{j+1}}\\
        &=& \sum_{j = 0}^n \frac{\|f-g\|_j}{2^{j+1}} + \sum_{j = n+1}^\infty \frac{\|f-g\|_j}{2^{j+1}}\\
        &\leq& \sum_{j = 0}^n \frac{\|f-g\|_n}{2^{j+1}} + \sum_{j = n+1}^\infty \frac{1}{2^{j+1}}\\
        &<& \|f-g\|_n + \frac{\varepsilon}{2}\\
        &\leq& \frac{\varepsilon}{2} + \frac{\varepsilon}{2} = \varepsilon.
    \end{eqnarray*}
    It follows that \([f;K_n,\varepsilon/2] \subseteq B_d(f;\varepsilon)\).

    \ref{metri} \(\Rightarrow\) \ref{firstly}:
    Obvious.

    \ref{firstly} \(\Rightarrow\) \ref{hemi}:
    Suppose \(C_k(X)\) is first-countable.
	Let \(\mathcal B = \{ U_n : n \in \omega \}\) be a descending neighborhood basis at \(\mathbf 0\).
	Then, for each \(n \in \omega\) let \(K_n \subseteq X\) be compact and \(q_n > 0\) be a rational number so that
	\[
		[\mathbf 0; K_n, q_n] \subseteq U_n.
	\]
	Suppose, toward a contradiction, that there is some compact \(K \subseteq X\) so that, for each \(n \in \omega\), \(K \not\subseteq K_n\).
	Then, for \(n \in \omega\), we can choose \(x_n \in K \setminus K_n\) and define \(f_n : X \to [0,1]\) so that \(f_n(x_n)=1\) and \(f_n[K_n] \equiv 0\).
	Since \(f_n \in [\mathbf 0; K_n , q_n]\) for each \(n\in\omega\), \(f_n \to \mathbf 0\).
	In particular, for any \(\varepsilon \in (0,1)\), there must be some \(n \in \omega\) so that
	\[
	    f_n \in [\mathbf 0; K, \varepsilon] \implies 1 = f_n(x_n) < \varepsilon,
    \]
    an absurdity.
    Hence, \(X\) is hemicompact.

    \ref{hemi} \(\Leftrightarrow\) \ref{prekKO}:
    This follows from Theorem \ref{thm:StrongPreCovering}.

    \ref{prekKO} \(\Leftrightarrow\) \ref{markKRoth}:
    This follows from Corollary \ref{cor:kKODual}.

    \ref{markKRoth} \(\Rightarrow\) \ref{markCkCFT}:
    Let \(t\) be a winning mark for Two in \(G_1(\mathcal K_X , \mathcal K_X)\).
    By Corollary \ref{cor:CompactOpenCofinal}, we can assume \(t\) is a winning mark for Two in \(G_1(\mathcal K_X , \Gamma_k(X))\)
    For any \(A \in \Omega_{C_k(X), \mathbf 0}\), let
    \[
        U_{A,n} = t(\mathscr U(A,n) , n)
    \]
    where
    \[
        \mathscr U(A,n) := \{ f^{-1}[(-2^{-n},2^{-n})] : f \in A \}.
    \]
    Note that \(\mathscr U(A,n)\) is a \(k\)-cover, so this definition makes sense. Say \(U_{A,n} = f_{A,n}^{-1}[(-2^{-n},2^{-n})]\) and set
    \[
        \tau(A,n) = f_{A,n}.
    \]
    We claim that \(\tau\) is a winning mark for Two in \(G_1(\Omega_{C_k(X), \mathbf 0}, \Omega_{C_k(X), \mathbf 0})\).
    Suppose \(A_0,f_{A_0,0},\cdots\) is a play of this game according to \(\tau\).
    Set \(f_n = f_{A_n,n}\) and \(U_n = U_{A_n,n}\).
    We need to show that \(\mathbf 0 \in \overline{\{f_n : n \in \omega\}}\).
    Let \(K \subseteq X\) be compact and \(\varepsilon > 0\).
    Then \(K \subseteq U_n\) for some \(n\) with \(2^{-n} < \varepsilon\).
    So \(f_n[K] \subseteq (-2^{-n},2^{-n}) \subseteq (-\varepsilon,\varepsilon)\) and thus \(f_n \in [\mathbf 0; K, \varepsilon]\).
    Therefore \(\mathbf 0 \in \overline{\{f_n : n \in \omega\}}\) and \(\tau\) is a winning mark for Two in \(G_1(\Omega_{C_k(X), \mathbf 0}, \Omega_{C_k(X), \mathbf 0})\).

    \ref{markCkCFT} \(\Rightarrow\) \ref{markCkCDFT}:
    This follows immediately from the fact that \(\mathcal D_{C_k(X)} \subseteq \Omega_{C_k(X), \mathbf 0}\).

    \ref{markCkCDFT} \(\Rightarrow\) \ref{markKRoth}:
    Let \(t\) be a winning mark for Two in \(G_1(\mathcal D_{C_k(X)}, \Omega_{C_k(X), \mathbf 0})\).
    Let
    \[
        D(\mathscr U) = \{g \in C_k(X) : (\exists U \in \mathscr U)(g[X \setminus U] \equiv 1)\}.
    \]
    We claim that \(D(\mathscr U)\) is dense in \(C_k(X)\).
    Indeed, let \(f \in C_k(X)\), \(K \in K(X)\), and \(\varepsilon > 0\).
    Choose \(U \in \mathscr U\) so that \(K \subseteq U\) and then let \(g \in C_k(X)\) be so that \(g|_K \equiv f|_K\) and \(g[X \setminus U] \equiv 1\).
    Notice that \(g \in D(\mathscr U) \cap [f; K , \varepsilon]\).
    Consequently, \(D(\mathscr U)\) is dense in \(C_k(X)\).

    Define \(\tau : \mathcal K_X \times \omega \to \mathscr T_X\) by letting \(\tau(\mathscr U, n)\) be \(U \in \mathscr U\) so that \(f[X\setminus U] \equiv 1\) where \(f = t(D(\mathscr U), n)\).
    Let \(\mathscr U_1 , \mathscr U_2\) be \(k\)-covers played by One and
    \[
        U_n = \tau(\mathscr U_n , n).
    \]
    Let \(K\) be compact.
    Then there is some \(n \in \omega\) so that \(f_n \in [\mathbf 0; K, 1/2]\).
    Since \(f_n[X \smallsetminus U_n] \equiv 1\), it must be the case that \(K \cap (X \smallsetminus U_n) = \emptyset\).
    Thus \(K \subseteq U_n\).

    \ref{firstly} \(\Rightarrow\) \ref{preGru}:
    Let \(\{U_n : n \in \omega\}\) be a descending sequence of open sets which form a base for the topology of \(C_k(X)\) at \(\mathbf 0\).
    Then let One play \(U_n\) in the \(n^{\text{th}}\) inning and notice that this defines a predetermined winning strategy for One in \(G_1(\mathscr N_{\mathbf 0}, \neg \Gamma_{C_k(X),\mathbf 0})\).

    \ref{preGru} \(\Rightarrow\) \ref{preCL} \(\Rightarrow\) \ref{preCD}:
    Obvious

    \ref{preCD} \(\Rightarrow\) \ref{hemi}:
    Let \(U_n = \sigma(n)\).
    Say \(K_n\) is the support of \(U_n\).
    We claim that \(\{K_n : n \in \omega\}\) witnesses that \(X\) is hemicompact.
    Towards a contradiction, suppose otherwise.
    Then there is a compact \(K \subseteq X\) so that \(K \not \subseteq K_n\) for any \(n\).
    Define \(f_n \in U_n\) with the property that \(f(x_n) = n\) where \(x_n \in K \setminus K_n\).
    Then \(U_0,f_0,\cdots\) is a play of \(G_1(\mathscr T_{C_k(x)}, CD_{C_k(X)})\) according to \(\sigma\), but \(\{f_n : n \in \omega\}\) has no limit points.
    Therefore Two has won the game, which is a contradiction.
\end{proof}

\begin{thm}
    For a Tychonoff space \(X\), the following are equivalent.
    \begin{enumerate}[label=(\alph*)]
	    \item \label{fullStratkKO}
	    \(\text{I} \uparrow G_1(\mathscr N[K(X)], \neg\mathcal K_X)\), a variant of the compact-open game
	    \item \label{Gruen}
	    \(\text{I} \uparrow G_1(\mathscr N_{\mathbf 0}, \neg \Gamma_{C_k(X),\mathbf 0})\), Gruenhage's W-game
	    \item \label{CL}
	    \(\text{I} \uparrow G_1(\mathscr N_{\mathbf 0}, \neg \Omega_{C_k(X),\mathbf 0})\), the closure game
	    \item \label{CD}
	    \(\text{I} \uparrow G_1(\mathscr T_{C_k(x)}, CD_{C_k(X)})\), the closed discrete selection game
	    \item \label{TwoSelection}
        \(\text{II} \uparrow G_1(\mathcal K_X, \mathcal K_X)\), a variant of the Rothberger game
        \item \label{TwoBladeSelect}
        \(\text{II} \uparrow G_1(\Omega_{C_k(X), \mathbf 0}, \Omega_{C_k(X), \mathbf 0})\), the countable fan tightness game
        \item \label{TwoDenseSelectBlade}
        \(\text{II} \uparrow G_1(\mathcal D_{C_k(X)}, \Omega_{C_k(X), \mathbf 0})\), the countable dense fan tightness game
	 \end{enumerate}
\end{thm}
\begin{proof}
    \ref{fullStratkKO} \(\Rightarrow\) \ref{Gruen}:
	Suppose One wins \(G_1(\mathscr N[K(X)], \neg \mathcal K_X)\).
	By Lemma \ref{lem:CompactOpenCofinal}, we know that One wins \(G_1(\mathscr N[K(X)], \neg \Gamma_k(X))\) so let \(s\) be a winning strategy for One in \(G_1(\mathscr N[K(X)], \neg \Gamma_k(X))\) and let \(K_0 = s(\emptyset)\).
	We define
	\[
		\sigma(\emptyset) = U_0 = [\mathbf 0; K_0, 1].
	\]
	For Two's response \(f_0 \in U_0\), let \(V_0 = f_0^{-1}[(-1,1)]\) and notice that \(K_0 \subseteq V_0\).

	For \(n \in \omega\), suppose we have a partial run
	\[
		\langle K_0 , V_0 , \ldots , K_n , V_n \rangle
	\]
	of \(G_1(\mathscr N[K(X)], \neg \mathcal K_X)\) and a partial run
	\[
		\langle U_0 , f_0 , \ldots , U_n , f_n \rangle
	\]
	of \(G_1(\mathscr N_{\mathbf 0}, \neg \Gamma_{C_k(X),\mathbf 0})\) and let
	\[
		K_{n+1} = s(K_0 , V_0 , \ldots , K_n , V_n).
	\]
	Define
	\[
		\sigma(U_0 , f_0 , \ldots , U_n , f_n) = U_{n+1} = \left[ \mathbf 0; \bigcup_{j=0}^{n+1} K_j , \frac{1}{2^{n}} \right].
	\]
	For Two's choice of \(f_{n+1} \in U_{n+1}\), define \(V_{n+1} = f_{n+1}^{-1}\left[ \left( -2^{-n} , 2^{-n} \right) \right]\) and notice that \(K_{n+1} \subseteq V_{n+1}\).
	This completes the definition of \(\sigma\).

	Now, we will show that \(f_n \to \mathbf 0\).
	Let \(K \subseteq X\) be an arbitrary compact set, \(\varepsilon > 0\), and consider \([\mathbf 0 ; K , \varepsilon]\).
	Choose \(n_0 \in \omega\) so that \(2^{-n_0} < \varepsilon\).
	By our choice of strategy, \(\{ V_n : n \geq n_0 \}\) is a \(k\)-cover of \(X\) so there must be some \(n_1 \geq n_0\) so that \(K \subseteq V_{n}\) for \(n \geq n_1\).
	That is, for any \(n \geq n_1\),
	\[
			K \subseteq V_{n+1} \implies (\forall x \in K)\left(|f_{n+1}(x)| < \frac{1}{2^n} \leq \frac{1}{2^{n_0}} < \varepsilon \right)
			\implies f_{n+1} \in [\mathbf 0 ; K,\varepsilon].
	\]
	Therefore \(f_n \to \mathbf 0\).

	\ref{Gruen} \(\Rightarrow\) \ref{CL} \(\Rightarrow\) \ref{CD}:
	Obvious.

	\ref{CD} \(\Rightarrow\) \ref{fullStratkKO}:
	Let \(s\) be a winning strategy for One in \(G_1(\mathscr T_{C_k(x)}, CD_{C_k(X)})\).
	For every non-empty open subset \(W \subseteq C_k(X)\), pick \(f \in W\), a compact \(K \subseteq X\), and \(\varepsilon \in (0,1)\) so that \([f;K,\varepsilon] \subseteq W\) and define \(\phi(W) = f\), \(\text{supp}(W) = K\), and \(\gamma(W) = [f;K,\varepsilon]\).

	Now we define a strategy \(\sigma\) for One as follows.
	To start, let \(W_0 = s(\emptyset)\) and \(\sigma(\emptyset) = K_0 = \text{supp}(W_0)\).
	For \(n \in \omega\), suppose we have
	\[
	    \langle K_0 , U_0 , K_1 , U_1 , \ldots , K_n \rangle \in \sigma
	\]
	defined along with
	\[
	    \langle W_0 , f_0 , W_1, f_1, \ldots , W_n \rangle \in s
	\]
	so that
	\begin{itemize}
	    \item
	    \((\forall j \leq n)(\text{supp}(W_j) = K_j)\),
	    \item
	    \((\forall j < n)(f_j \in \gamma(W_j))\), and
	    \item
	    \((\forall j < n)(f_j[X \setminus U_j] \equiv j)\).
	\end{itemize}
	Let \(U_n \subseteq X\) be open so that \(K_n \subseteq U_n\).
	As \(K_n\) is compact, we can find \(g_n : X \to [0,1]\) continuous so that \(g_n[K_n] \equiv 1\) and \(g_n[X\setminus U_n] \equiv 0\).
	Then, let
	\[
	    f_n = \phi(W_n) \cdot g_n + (1-g_n) \cdot n
	\]
	and observe that
	\begin{itemize}
	    \item
	    \(f_n|_{K_n} \equiv \phi(W_n)|_{K_n} \implies f_n \in \gamma(W_n)\) and
	    \item
	    \(f_n[X \setminus U_n] \equiv n\).
	\end{itemize}
	Then let
	\[
	    W_{n+1} = s(W_0 , f_0 , W_1, f_1, \ldots , W_n , f_n)
	\]
	and
	\[
	    \sigma(K_0 , U_0 , K_1 , U_1 , \ldots , K_n , U_n) = K_{n+1} := \text{supp}(W_{n+1}).
	\]
	This completes the definition of \(\sigma\).

	Now, we will show that \(\sigma\) is a winning strategy for One in \(G_1(\mathscr N[K(X)], \neg \mathcal K_X)\).
	Toward this end, suppose \(K_0 , U_0 , K_1 , U_1 , \ldots\) is a play of \(G_1(\mathscr N[K(X)], \neg \mathcal K_X)\) according to \(\sigma\) along with the corresponding play \(W_0 , f_0 , W_1 , f_1 , \ldots\) of \(G_1(\mathscr T_{C_k(x)}, CD_{C_k(X)})\) according to \(s\).
	Since \(s\) is winning in \(G_1(\mathscr T_{C_k(x)}, CD_{C_k(X)})\), let \(f \in C_k(X)\) be an accumulation point of \(\{f_n : n \in \omega\}\).

	By way of contradiction, assume there is a compact \(K \subseteq X\) so that \(K \not\subseteq U_n\) for all \(n \in \omega\) and consider the neighborhood \([f;K,1]\) of \(f\).
	As \(f\) is continuous, \(f[K]\) is bounded so choose \(n \in \omega\) so that \(|f(x)| < n\) for all \(x \in K\).
	Now, for any \(m > n\), we can find \(x \in K \setminus U_m\) which necessitates
	\[
	    1 \leq m - n = f_m(x) - n < f_m(x) - f(x) = |f_m(x) - f(x)|.
	\]
	In other words, \(f_m \not\in [f;K,1]\) for all \(m > n\), which is a contradiction.
	Therefore, One wins \(G_1(\mathscr N[K(X)], \neg \mathcal K_X)\).

	\ref{fullStratkKO} \(\Leftrightarrow\) \ref{TwoSelection}:
    This follows from Corollary \ref{cor:kKODual}.

    \ref{TwoSelection} \(\Rightarrow\) \ref{TwoBladeSelect}:
    For \(A \in \Omega_{C_k(X), \mathbf 0}\), recall that
    \[
        \mathscr U(A,n) := \{ f^{-1}[(-2^{-n},2^{-n})] : f \in A \}
    \]
    is a \(k\)-cover of \(X\).
    From Lemma \ref{lem:CompactOpenCofinal}, we know that \(\text{I} \uparrow G_1(\mathscr N[K(X)], \neg\mathcal K_X)\) if and only if \(\text{I} \uparrow G_1(\mathscr N[K(X)], \neg\Gamma_k(X))\).
    By Corollary \ref{cor:GeneralGalvin}, the games \(G_1(\mathscr N[K(X)], \neg\Gamma_k(X))\) and \(G_1(\mathcal K_X, \Gamma_k(X))\) are dual.
    Hence, we let \(t\) be a winning strategy for Two in \(G_1(\mathcal K_X, \Gamma_k(X))\).

    We now define \(\tau\).
    Given \(A_0 \in \Omega_{C_k(X), \mathbf 0}\), let \(U_0 = t(\mathscr U(A_0, 0))\).
    Choose \(f_0 \in A_0\) so that \(U_0 = f_0^{-1}[(-1,1)]\) and let \(\tau(A_0) = f_0\).

    Suppose we have
    \begin{itemize}
        \item
        \(A_0, A_1, \ldots, A_n \in \Omega_{C_k(X), \mathbf 0}\),
        \item
        \(U_{j+1} = t(\mathscr U(A_0, 0) , U_0, \mathscr U(A_1, 1) , U_1, \ldots, \mathscr U(A_{j}, j))\), and
        \item \(f_j \in A_j\) so that \(U_j = f_j^{-1}[(-2^{-j},2^{-j})]\).
    \end{itemize}
    Further suppose that \(\tau(A_0,\cdots,A_j) = f_j\).
    Given \(A_{n+1} \in \Omega_{C_k(X),\mathbf 0}\), let
    \[
        U_{n+1} = t(\mathscr U(A_0,0) , U_0 , \mathscr U(A_1, 1) ,  U_1 , \ldots , \mathscr U(A_{n}, n) ,  U_{n} , \mathscr U(A_{n+1},n+1))
    \]
    Then choose \(f_{n+1} \in A_{n+1}\) so that \(U_{n+1} = f_{n+1}^{-1}[(-2^{-(n+1)} , 2^{-(n+1)})]\) and let
    \[
        \tau(A_0 , f_0 , A_1 , f_1 , \ldots A_n , f_n , A_{n+1}) = f_{n+1}.
    \]
    This finishes the definition of \(\tau\).

    To see that \(\tau\) is winning for Two, consider a full run
    \[
        A_0 , f_0 , A_1 , f_1 , \ldots
    \]
    of the game \(G_1(\Omega_{C_k(X), \mathbf 0}, \Omega_{C_k(X), \mathbf 0})\) played according to \(\tau\) and consider the corresponding run
    \[
        \mathscr U(A_0,0) , U_0 , \mathscr U(A_1,1) , U_1 , \ldots
    \]
    of the game \(G_1(\mathcal K_X, \Gamma_k(X))\) played according to \(t\).
    We have that \(\{ U_n : n \in \omega\} \in \Gamma_k(X)\).

    We wish to show that \(\{ f_n : n \in \omega \} \in \Omega_{C_k(X) , \mathbf 0}\).
    So let \(K \in K(X)\) and \(\varepsilon > 0\) be arbitrary.
    Then pick \(n \in \omega\) so that \(2^{-n} < \varepsilon\).
    Observe that there exists \(m \geq n\) so that
    \[
        K \subseteq U_m = f_{m}^{-1}[(-2^{-m},2^{-m})].
    \]
    Then \(f_m \in [\mathbf 0; K , \varepsilon]\) establishing that \(\{ f_n : n \in \omega \} \in \Omega_{C_k(X) , \mathbf 0}\).

    \ref{TwoBladeSelect} \(\Rightarrow\) \ref{TwoDenseSelectBlade}:
    This follows from the fact that \(\mathcal D_{C_k(X)} \subseteq \Omega_{C_k(X) , \mathbf 0}\).

    \ref{TwoDenseSelectBlade} \(\Rightarrow\) \ref{TwoSelection}:
    Let \(t\) be a winning strategy for Two in \(G_1(\mathcal D_{C_k(X)}, \Omega_{C_k(X), \mathbf 0})\).
    For \(\mathscr U \in \mathcal K_X\), recall that
    \[
        D(\mathscr U) = \{g \in C_k(X) : (\exists U \in \mathscr U)(g[X \setminus U] \equiv 1)\}
    \]
    is dense in \(C_k(X)\).

    We define \(\tau\).
    Given \(\mathscr U_0 \in \mathcal K_X\), let \(f_0 = t(D(\mathscr U_0))\).
    Then choose \(U_0 \in \mathscr U_0\) so that \(f_0[X \setminus U_0] \equiv 1\) and define \(\tau(\mathscr U_0) = U_0\).
    Inductively, given \(\mathscr U_{n+1} \in \mathcal K_X\), let
    \[
        f_{n+1} = t(D(\mathscr U_0), f_0, D(\mathscr U_1) , f_1, \ldots, D(\mathscr U_n))
    \]
    and choose \(U_{n+1} \in \mathscr U_{n+1}\) so that \(f_{n+1}[X \setminus U_{n+1}] \equiv 1\).
    Set
    \[
        \tau(\mathscr U_0, U_0, \mathscr U_1, U_1, \ldots, \mathscr U_{n+1}) = U_{n+1}.
    \]

    To see that \(\{U_n : n \in \omega\}\) is a \(k\)-cover of \(X\), let \(K \in K(X)\).
    Then there must exist \(n \in \omega\) so that \(f_n \in [\mathbf 0; K , 1/2]\).
    Since \(f_n[X \setminus U_n] \equiv 1\), it must be the case that \(K \cap (X \setminus U_n) = \emptyset\); i.e. \(K \subseteq U_n\).
    Therefore \(\tau\) is a winning strategy.
\end{proof}

\begin{example}
    Give \(\omega_1\) the discrete topology and let \(L(\omega_1)\) be the one-point Lindel{\"{o}}fication of \(\omega_1\).
    Then \(L(\omega_1)\) is not countable, and so player One does not have a pre-determined strategy in the finite-open game on \(L(\omega_1)\).
    However, player One does have a winning strategy in the finite-open game on \(L(\omega_1)\) for, let One play the point at infinity first.
    Then Two plays a co-countable open set and then One need only play an enumeration of that set.
    Since every compact subset of \(L(\omega_1)\) must be finite, \(L(\omega_1)\) is an example of a space where \(\text{I} \uparrow G_1(\mathscr N[K(X)], \neg\mathcal K_X)\), but \(\text{I} \not\underset{\text{pre}}{\uparrow} G_1(\mathscr N[K(X)], \neg\mathcal K_X)\).

    We have not yet been able to find a space \(X\) where
    \begin{itemize}
        \item \(\text{I} \not\uparrow G_1(\mathscr N[X^{<\omega}],\neg \mathcal O_X)\) (the finite-open game),
        \item \(\text{I} \underset{\text{pre}}{\not\uparrow} G_1(\mathscr N[K(X)], \neg \mathcal K_X)\), and
        \item \(\text{I} \uparrow G_1(\mathscr N[K(X)], \neg\mathcal K_X)\).
    \end{itemize}
\end{example}

\subsection{Selection Principles}

Propositions \ref{prop:Pawlikowski1} and \ref{prop:Pawlikowski2} are generalizations of Pawlikowski's results \cite{Pawlikowski1994}.

Recall that a collection \(\mathcal A\) is a \emph{refinement} of a collection \(\mathcal B\), denoted \(\mathcal A \preceq \mathcal B\), if \((\forall A \in \mathcal A)(\exists B \in \mathcal B)[A \subseteq B]\).

\begin{prop}
    \label{prop:Pawlikowski1}
    {\color{red}
    Suppose \(\mathcal A \preceq \mathcal B\) and \(S_{\text{fin}}(\mathcal O(X,\mathcal A), \mathcal O(X, \mathcal B))\).
    Then \(\text{I} \not\uparrow G_{\text{fin}}(\mathcal O(X,\mathcal A), \Lambda(X, \mathcal B))\).
    Thus,
    \[
    \text{I} \underset{\text{pre}}{\uparrow} G_{\text{fin}}(\mathcal O(X,\mathcal A), \mathcal O(X, \mathcal B)) \Leftrightarrow \text{I} \uparrow G_{\text{fin}}(\mathcal O(X,\mathcal A), \mathcal O(X, \mathcal B)).
    \]
    }
    Correction:
    \[
    \text{I} \underset{\text{pre}}{\uparrow} G_{\text{fin}}(\mathcal K_X, \mathcal K_X) \Leftrightarrow \text{I} \uparrow G_{\text{fin}}(\mathcal K_X, \mathcal K_X).
    \]
    The more general statement may be true, but we do not have a proof of it at this time.
    For a proof of the statement for \(k\)-covers, see \href{https://arxiv.org/abs/2102.00296}{arXiv:2102.00296}.
\end{prop}
\begin{proof}
    {\color{red} We leave the original proof here for the record and note the exact place it goes wrong in red.}
    Suppose \(S_{\text{fin}}(\mathcal O(X,\mathcal A),\mathcal O(X,\mathcal B))\) and that One plays according to a fixed strategy.
    By \(S_{\text{fin}}(\mathcal O(X,\mathcal A),\mathcal O(X,\mathcal B))\), any \(\mathscr U \in \mathcal O(X,\mathcal A)\) contains a countable subcover \(\{U_n:n\in\omega\} \in \mathcal O(X,\mathcal B)\).
    Since Two gets to choose finitely many open sets, Two can respond to One's move with an initial segment of \(\{U_n : n \in \omega\}\).
    {\color{red}Hence, without loss of generality, we may assume that One is playing countable covers and, by taking unions, that One's plays consist of ascending open sets; i.e., that \(U_n \subseteq U_{n+1}\).
    Because we are talking about covering sets and not points, we cannot replace the open sets in the cover with finite unions.
    For example, \[\{B(q_1,1) \cup \cdots \cup B(q_n,1) : q_1 , \ldots , q_n \in \mathbb Q\}\] is a \(k\)-cover of \(\mathbb R\) but \(\{B(q,1) : q \in \mathbb Q\}\) is not a \(k\)-cover of \(\mathbb R\).}
    Then, Two's response can be coded by a natural number.
    Thus, we can identify a strategy for One with a family \(\{U_s : s \in \omega^{<\omega} \}\) satisfying
    \begin{itemize}
        \item
        for each \(s\in \omega^{<\omega}\), \(\{ U_{s \concat k} : k \in \omega\} \in \mathcal O(X,\mathcal B)\) and
        \item
        for each \(s\in \omega^{<\omega}\) and \(j \in \omega\), \(U_{s \concat j} \subseteq U_{s \concat (j+1)}\).
    \end{itemize}
    We can also identify the response by Two as a function from \(\omega\) to \(\omega\).

    For integers \(j > 0\) and \(k,m \geq 0\), define
    \[
        V_k(m,j) = \bigcap \{ U_{s \concat k} : s \in j^m \}.
    \]
    We claim that, for all \(m \geq 0\) and \(j > 0\), \(\{ V_k(m,j) : k \in \omega \} \in \mathcal O(X, \mathcal B)\) and \(V_{k}(m,j) \subseteq V_{k+1}(m,j)\).
    Fix \(m \geq 0\), \(j > 0\), and \(B \in \mathcal B\).
    For \(s \in j^m\), \(\{U_{s \concat k} : k \in \omega \} \in \mathcal O(X,\mathcal B)\) so let \(k_s \in \omega\) be so that \(B \subseteq U_{s \concat k_s}\).
    Set \(k^\ast = \max\{ k_s : s \in j^m \}\).
    By monotonicity, \(B \subseteq U_{s \concat k^\ast}\).
    Since \(s\) was arbitrary, we see that \(B \subseteq U_{s \concat k^\ast}\) for all \(s \in j^m\) and, therefore, \(B \subseteq V_{k^\ast}(m,j)\).

    For integers \(j > 0\), \(k \geq j\), and \(m,n \geq 0\), define
    \[
        W_k^n(m,j) = \bigcup \left\{ \bigcap_{i = 0}^n V_{k_{i+1}}(m+i,k_i) : j = k_0 \leq k_1 \leq \cdots \leq k_{n+1} = k \right\}
    \]
    We claim that, for all \(m, n \geq 0\), and \(j > 0\), \(\{ W_k^n(m,j) : k \geq j\} \in \mathcal O(X,\mathcal B)\) and that \(W_k^n(m,j) \subseteq W_{k+1}^n(m,j)\).
    So fix \(j = k_0 > 0\), \(m,n \geq 0\), and \(B \in \mathcal B\).
    Let \(k_1 \geq j\) be so that \(B \subseteq V_{k_1}(m,j)\).
    For \(i \geq 1\), let \(k_{i+1} \geq k_i\) be so that \(B \subseteq V_{k_{i+1}}(m+i,k_i)\).
    Then, let \(k = k_{n+1}\) and observe that
    \[
        B \subseteq \bigcap_{i=0}^n V_{k_{i+1}}(m+i,k_i) \subseteq W_k^n(m,j).
    \]
    That is, \(\{ W_k^n(m,j) : k \geq j\} \in \mathcal O(X,\mathcal B)\).

    Let \(\{A_\ell:\ell\in\omega\}\) be a partition of \(\omega\) consisting of infinite sets.
    Then, we can apply \(S_{\text{fin}}(\mathcal O(X,\mathcal A),\mathcal O(X,\mathcal B))\) to \( \{\{ W_k^n(m,j) : k \geq j \} : n \in A_\ell\}\) to obtain \(f_{m,j,\ell} : A_\ell \to \omega\) so that \(\{ W_{f_{m,j,\ell}(n)}^n(m,j) : n \in A_\ell \} \in \mathcal O(X,\mathcal B)\).
    Then let \(f_{m,j} = \bigcup_{\ell \in n} f_{m,j,\ell}\) and observe that \(f_{m,j}\) has the property that, for any \(B \in \mathcal B\), \(B \subseteq W_{f_{m,j}(n)}^n(m,j)\) for infinitely many \(n \in \omega\).

    Now we define \(f : \omega \to \omega\).
    Let \(f(0) = \max\{f_{0,1}(0),1\}\).
    Fix \(n \geq 1\) and suppose \(f(i)\) has been defined for all \(i < n\).
    Let
    \[
        f(n) = 1 + \max\left( \{f(i) : i < n\} \cup \{ f_{m,j}(i) : m,j,i \leq n \} \right).
    \]
    Let \(m \geq 0\), \(j > 0\) and \(n \geq \max\{m,j\}\).
    Then observe that \(f(m+n) \geq f(n) \geq f_{m,j}(n)\).
    For \(B \in \mathcal B\), notice that
    \[
        B \subseteq W_{f_{m,j}(n)}^n(m,j) \subseteq W_{f(m+n)}^n(m,j).
    \]

    Toward a contradiction, suppose there is \(B\in\mathcal B\) and \(m \in \omega\) so that \(B \not\subseteq U_{f\restriction_{m+n}}\) for all \(n \in \omega\).
    Notice that \(f(m) > 0\) so there must be some \(n \geq \max\{m,f(m)\}\) so that
    \[
        B \subseteq W_{f_{m,f(m)}(n)}^n(m,f(m)) \subseteq W_{f(m+n)}^n(m,f(m)).
    \]
    Now, there must be \(f(m) = k_0 \leq k_1 \leq \cdots \leq k_{n+1} = f(m+n)\) so that
    \[
        B \subseteq \bigcap_{i=0}^n V_{k_{i+1}}(m+i,k_i).
    \]
    Note that \(f\restriction_m \in k_0^m\), \(B \subseteq V_{k_1}(m,k_0)\), and \(B \not\subseteq U_{f\restriction_{m+1}}\).
    Similarly, \(f \restriction_{m+1} \in k_1^{m+1}\), \(B \subseteq V_{k_2}(m+1,k_1)\), and \(B \not\subseteq U_{f\restriction_{m+2}} = U_{f\restriction_{m+1} \concat f(m+1)}\).
    \[
        B \subseteq V_{k_2}(m+1, k_1) = \bigcap \{ U_{s \concat k_2} : s \in k_1^{m+1} \} \subseteq U_{f\restriction_{m+1} \concat k_2} \subseteq U_{f\restriction_{m+1} \concat \ell}
    \]
    where \(\ell \geq k_2\).
    Since \(B \not\subseteq U_{f\restriction_{m+1} \concat f(m+1)}\), it must be the case that \(f(m+1) < k_2\).
    Continue this way to see that \(f(m+n) < k_{n+1}\), a contradiction.
\end{proof}

\begin{cor}
    Suppose \(S_{\text{fin}}(\mathcal K_X, \mathcal K_X)\).
    Then One has no winning strategy in \(G_{\text{fin}}(\mathcal K_X, \mathcal K_X)\).
    Also, \(\text{I} \uparrow G_{\text{fin}}(\mathcal K_X, \mathcal K_X)\) if and only if \(\text{I} \underset{\text{pre}}{\uparrow} G_{\text{fin}}(\mathcal K_X, \mathcal K_X)\).
\end{cor}

\begin{prop}
    \label{prop:Pawlikowski2}
    {\color{red}
    Suppose \(\mathcal A \preceq \mathcal B\) and \(S_1(\mathcal O(X, \mathcal A), \mathcal O(X, \mathcal B))\).
    Then \(\text{I} \not\uparrow G_1(\mathcal O(X, \mathcal A), \Lambda (X, \mathcal B))\).
    Thus
    \[
    \text{I} \underset{\text{pre}}{\uparrow} G_1(\mathcal O(X,\mathcal A), \mathcal O(X, \mathcal B)) \Leftrightarrow \text{I} \uparrow G_1(\mathcal O(X,\mathcal A), \mathcal O(X, \mathcal B)).
    \]
    }
    Correction:
    \[
    \text{I} \underset{\text{pre}}{\uparrow} G_{1}(\mathcal K_X, \mathcal K_X) \Leftrightarrow \text{I} \uparrow G_{1}(\mathcal K_X, \mathcal K_X).
    \]
    The more general statement may be true, but we do not have a proof of it at this time.
    For a proof of the statement for \(k\)-covers, see \href{https://arxiv.org/abs/2102.00296}{arXiv:2102.00296}.
\end{prop}
\begin{proof}
    {\color{red} We leave the original proof here for the record and note the exact place it goes wrong in red.}
    Suppose \(S_{1}(\mathcal O(X,\mathcal A),\mathcal O(X,\mathcal B))\) and that One is playing according to a fixed strategy.
    By the selection principle, every \(\mathscr U \in \mathcal O(X,\mathcal A)\) has a countable sub-cover \(\mathscr V \in \mathcal O(X,\mathcal B)\).
    So we can code One's strategy with \(\{ U_s : s \in \omega^{<\omega} \}\) which as the property that, for any \(s\in \omega^{<\omega}\), \(\{ U_{s \concat k} : k \in \omega \} \in \mathcal O(X,\mathcal B) \subseteq \mathcal O(X,\mathcal A)\) by \(\mathcal A \preceq \mathcal B\).
    A play by Two can be represented with \(f : \omega \to \omega\).
    We will show the existence of \(f : \omega \to \omega\) so that, for any \(B \in \mathcal B\), there are infinitely many \(n \in \omega\) for which \(B \subseteq U_{f \restriction_n}\).

    For \(m \geq 0\), \(j > 0\), and \(s : |j^m| \to \omega\), define
    \[
        U_s(m,j) = \bigcap \left\{ \bigcup \{ U_{t \concat (s \restriction_k)} : 0 < k \leq |j^m| \} : t \in j^m \right\}.
    \]
    If Two plays numbers \(< j\) on all rounds prior to the \(m^{\text{th}}\) round, and play according to \(s\) afterwards, then they are guaranteed to cover \(U_{t \concat (s\restriction_k)}\) at round \(m + |j^m|\).

    We claim that, for any integers \(m \geq 0\) and \(j > 0\),
    \[
        \{ U_s(m,j) : s \in \omega^{|j^m|} \} \in \mathcal O(X, \mathcal B).
    \]
    So let \(m \geq 0\), \(j > 0\), and \(B \in \mathcal B\).
    Enumerate \(j^m\) as \(\{ t_k : k < |j^m| \}\).
    We recursively define \(s : |j^m| \to \omega\) so that \(B \subseteq U_s(m,j)\).
    Choose \(s(0)\) so that \(B \subseteq U_{t_0 \concat s(0)}\).
    For \(n \geq 0\), suppose we have \(s(\ell)\) defined for all \(\ell \leq n\) so that
    \[
        B \subseteq U_{t_\ell \concat s(0) \concat s(1) \concat \cdots \concat s(\ell)}.
    \]
    Now define \(s(n+1)\) to be so that
    \[
        B \subseteq U_{t_{n+1} \concat s(0) \concat s(1) \concat \cdots \concat s(n+1)}.
    \]
    This defines \(s : |j^m| \to \omega\).
    Note that \(B \subseteq U_s(m,j)\).

    Next we claim that there are increasing functions \(g,h : \omega \to \omega\) so that whenever \(B \in \mathcal B\), there are infinitely many \(n \in \omega\) for which there exists \(s : (g(n+1) - g(n)) \to h(n+1)\) so that \(B \subseteq U_s(g(n), h(n))\).
    In symbols,
    \[
        (\forall B \in \mathcal B)(\exists^\infty n \in \omega)(\exists s : (g(n+1) - g(n)) \to h(n+1))[ B \subseteq U_s(g(n), h(n)) ].
    \]
    To do this, we define a strategy for One in \(G_{\text{fin}}(\mathcal O(X,\mathcal A), \Lambda (X,\mathcal B))\) so that Two's responses produce monotonically functions \(\omega \to \omega\).
    Define
    \[
        \mathscr U_{m,j} = \{ U_s(m,j) : s \in \omega^{|j^{m}|} \} \in \mathcal O(X,\mathcal B) \subseteq \mathcal O(X, \mathcal A)
    \]
    for \(m \geq 0\) and \(j > 1\).
    Also, define
    \[
        \mathscr F_{m,j,p} := \left\{ U_s(m,j) : s \in p^{|j^m|}\right\} \in \left[ \left\{ U_s(m,j) : s \in \omega^{|j^m|}\right\} \right]^{<\omega}
    \]
    where \(m \geq 0\), \(j > 0\), and \(p > 0\).
    Observe that, if \(p \leq q\),
    \[
        \mathscr F_{m,j,p} \subseteq \mathscr F_{m,j,q}.
    \]
    In fact, we can define
    \[
        p_{m,j} : \left[ \mathscr U_{m,j} \right]^{<\omega} \to \omega
    \]
    in the following way.
    For
    \[
        \mathcal E \in \left[ \mathscr U_{m,j} \right]^{<\omega},
    \]
    let \(I \in \left[ \omega^{|j^m|} \right]^{<\omega}\) be so that
    \[
        \mathcal E = \{ U_{s}(m,j) : s \in I \}
    \]
    and define
    \[
        p_{m,j}(\mathcal E) = 1 + \max\{ s(\ell) : s \in I \}
    \]
    We see that
    \[
        \mathcal E \subseteq \mathscr F_{m,j,p_{m,j}(\mathcal E)}.
    \]
    Set \(j_0 = 1\), \(m_0 = 0\), and
    \[
        \sigma(\emptyset) = \mathscr U_{m_0,j_0}.
    \]

    Now, for \(n \geq 0\), suppose \(\mathcal E_0 , \ldots , \mathcal E_{n-1}\), \(j_0 , j_1 , \ldots, j_n\), and \(m_0, m_1 , \ldots , m_n\) have been defined.
    Let \(m_{n+1} = m_n + |j_n^{m_n}|\).

    Then, for any
    \[
        \mathcal E_n \in \left[ \mathscr U_{m_n,j_n} \right]^{<\omega},
    \]
    let \(j_{n+1} = \max\{j_n, p_{m_n,j_n}(\mathcal E_n)\}\), and set
    \[
        \left\langle \mathscr U_{m_0,j_0} , \mathcal E_0 , \ldots , \mathscr U_{m_n,j_n} , \mathcal E_n , \mathscr U_{m_{n+1},j_{n+1}}  \right\rangle \in \sigma.
    \]
    This finishes the definition of \(\sigma\).
    {\color{red}
    Since One doesn't have any winning strategies in \(G_{\text{fin}}(\mathcal O(X,\mathcal A), \Lambda (X,\mathcal B))\) by Proposition \ref{prop:Pawlikowski1} (We have not established this fact at this point)}, there must be some counterplay by Two which results in \(m,j : \omega \to \omega\) so that
    \[
        \left \{ \mathscr F_{m_n,j_n,j_{n+1}} : n \in \omega \right \} \in \Lambda (X,\mathcal B).
    \]
    Then \(g = m\) and \(h = j\) are as desired.

    For any \(k_1 < \cdots < k_n\) and \(s_i : (g(k_{i} + 1) - g(k_i)) \to h(k_{i} + 1)\) define
    \[
        W_n(k_1,\ldots,k_n;s_1,\ldots,s_n) = \bigcap_{i=1}^n U_{s_i}(g(k_i),h(k_i)).
    \]
    Observe that \(g\) and \(h\) have been chosen so that, for any \(n \in \omega\),
    \[
        \left\{ W_n(k_1,\ldots,k_n;s_1,\ldots,s_n) : (k_1 < \cdots < k_n){\&}(\forall 1 \leq i \leq n)\left[ s_i \in h(k_{i}+1)^{(g(k_{i}+1) - g(k_i))} \right] \right\}
    \]
    is an element of \(\mathcal O(X,\mathcal B)\) and hence, an element of \(\mathcal O(X,\mathcal A)\).
    So we apply \(S_1(\mathcal O(X,\mathcal A),\mathcal O(X,\mathcal B))\) to find \(k_{i,n}\) and \(s_{i,n}\) for \(n\in\omega\) and \(i \leq n\) so that
    \[
        \{W_{n}(k_{1,n}, \ldots , k_{n,n}; s_{1,n}, \ldots , s_{n,n}) : n \in \omega \} \in \mathcal O(X,\mathcal B)
    \]
    and, moreover, there are infinitely many \(n\) so that
    \[
        B \subseteq W_{n}(k_{1,n}, \ldots , k_{n,n}; s_{1,n}, \ldots , s_{n,n})
    \]
    for any \(B \in \mathcal B\).

    Now take \(\ell_n \in \{k_{1,n} , \ldots , k_{n,n} \} \setminus \{ \ell_j : j < n \}\), say \(\ell_n = k_{i,n}\) and set \(t_n = s_{i,n}\).
    Notice that the \(\ell_n\) are all distinct and that, whenever \(B\in \mathcal B\), there are infinitely many \(n\) so that
    \[
        B \subseteq U_{t_n}(g(\ell_n) , h(\ell_n)).
    \]
    Define \(f : \omega \to \omega\) by
    \[
        f(g(\ell_n) + i) = t_n(i)
    \]
    for \(n \in \omega\) and \(i \in \text{dom}(t_n)\).
    Notice that the sets \(\{g(\ell_n) + i : i \in \text{dom}(t_n) \}\) are pairwise disjoint.
    For other \(n\in \omega\), set \(f(n) = 0\).

    We claim that \(f\) is as desired.
    Let \(B\in \mathcal B\) and notice that, if \(B \subseteq U_{t_n}(g(\ell_n),h(\ell_n))\), then
    \[
        B \subseteq \bigcup \left\{ U_{(f\restriction_{g(\ell_n)}) \concat (t_n\restriction_i)} : 0 \leq i \leq g(\ell_n + 1) - g(\ell_n) \right \}
    \]
    as \(f\restriction_{g(\ell_n)} : g(\ell_n) \to h(\ell_n)\).
    Now, \((f\restriction_{g(\ell_n)}) \concat (t_n\restriction_i) = f\restriction_{g(\ell_n) + i}\).
    {\color{red}So \(B \subseteq U_{f\restriction_{g(\ell_n) + i}}\) for some \(i\) with \(0 \leq i \leq g(\ell_n + 1) - g(\ell_n)\).
    Note that \(B\) can be a subset of a union of open sets without being completely contained in one of the open sets constituting the union.}
    Since this happened for infinitely many \(n\), \(f\) is as desired and we are done.
\end{proof}

\begin{cor}\label{cor:fullPre}
    Suppose \(S_1(\mathcal K_X, \mathcal K_X)\).
    Then One has no winning strategy in \(G_1(\mathcal K_X, \mathcal K_X)\).
    Also, \(\text{I} \uparrow G_1(\mathcal K_X, \mathcal K_X)\) if and only if \(\text{I} \underset{\text{pre}}{\uparrow} G_1(\mathcal K_X, \mathcal K_X)\)
\end{cor}

\begin{thm} \label{thm:FinalBoss}
    The following are equivalent for all Tychonoff spaces.
    \begin{enumerate}[label=(\alph*)]
        \item \label{TwoWinkKO}
        \(\text{II} \uparrow G_1(\mathscr N[K(X)], \neg \mathcal K_X)\), a variation of the compact-open game
        \item \label{TwoMarkkKO}
        \(\text{II} \underset{\text{mark}}{\uparrow} G_1(\mathscr N[K(X)], \neg \mathcal K_X)\)
        \item \label{OneWinkRoth}
        \(\text{I} \uparrow G_1(\mathcal K_X, \mathcal K_X)\), a variation of the Rothberger game
        \item \label{OnePreWinkRoth}
        \(X\) fails \(S_1(\mathcal K_X, \mathcal K_X)\), that is \(\text{I} \underset{\text{pre}}{\uparrow} G_1(\mathcal K_X, \mathcal K_X)\)
        \item \label{OneWinOmega}
        \(\text{I} \uparrow G_1(\Omega_{C_k(X),\mathbf{0}},\Omega_{C_k(X),\mathbf{0}})\),  the countable fan tightness game
        \item \label{OnePreWinOmega}
        \(C_k(X)\) is not sCFT, that is \(\text{I} \underset{\text{pre}}{\uparrow} G_1(\Omega_{C_k(X),\mathbf{0}},\Omega_{C_k(X),\mathbf{0}})\)
        \item \label{OneWinD}
        \(\text{I} \uparrow G_1(\mathcal{D}_{C_k(X)},\Omega_{C_k(X),\mathbf{0}})\), the countable dense fan tightness game
        \item \label{OnePreWinD}
        \(C_k(X)\) is not sCDFT, that is \(\text{I} \underset{\text{pre}}{\uparrow} G_1(\mathcal{D}_{C_k(X)},\Omega_{C_k(X),\mathbf{0}})\)
        \item \label{TwoWinGru}
        \(\text{II} \uparrow G_1(\mathscr N_{\mathbf 0}, \Omega_{C_k(X) , \mathbf 0})\), Gruenhage's clustering game
        \item \label{TwoMarkGru}
        \(\text{II} \underset{\text{mark}}{\uparrow} G_1(\mathscr N_{\mathbf 0}, \Omega_{C_k(X) , \mathbf 0})\)
        \item \label{TwoWinCL}
        \(\text{II} \uparrow G_1(\mathscr N_{\mathbf 0}, \Omega_{C_k(X) , \mathbf 0})\), the closure game
        \item \label{TwoMarkCL}
        \(\text{II} \underset{\text{mark}}{\uparrow} G_1(\mathscr N_{\mathbf 0}, \Omega_{C_k(X) , \mathbf 0})\)
        \item \label{TwoWinCD}
        \(\text{II} \uparrow G_1(\mathscr T_{C_k(X)} , CD_{C_k(X)})\), the closed discrete selection game
        \item \label{TwoMarkCD}
        \(\text{II} \underset{\text{mark}}{\uparrow} G_1(\mathscr T_{C_k(X)} , CD_{C_k(X)})\)
    \end{enumerate}
\end{thm}
\begin{proof}
    The equivalences \ref{TwoWinkKO} \(\Leftrightarrow\) \ref{OneWinkRoth} and \ref{TwoMarkkKO} \(\Leftrightarrow\) \ref{OnePreWinkRoth} are established by Corollary \ref{cor:kKODual}.
    Corollary \ref{cor:fullPre} establishes the equivalence of \ref{OnePreWinkRoth} and \ref{OneWinkRoth}.
    Thus \ref{TwoWinkKO}, \ref{OneWinkRoth}, \ref{TwoMarkkKO}, and \ref{OnePreWinkRoth} are equivalent.

    The equivalence of \ref{OnePreWinkRoth} and \ref{OnePreWinOmega} is Theorem 2.2 of \cite{Kocinac}.

    \ref{OnePreWinkRoth} \(\Rightarrow\) \ref{OnePreWinD}:
    This is an adaptation of \cite{ClontzRelatingGames} and is done by contrapositive.
    That is, suppose \(S_1(\mathcal D_{C_k(X),\mathbf{0}},\Omega_{C_k(X),\mathbf{0}})\). We will show \(S_1(\mathcal K_X, \mathcal K_X)\).
    For \(\mathscr U \in \mathcal K_X\), recall that
    \[
        D(\mathscr U) = \{g \in C_k(X) : (\exists U \in \mathscr U)(g[X \setminus U] \equiv 1)\}
    \]
    is dense in \(C_k(X)\) as seen in the proof of Theorem \ref{thm:Hemicompact}.

    Now, let \(\{ \mathscr U_n : n \in \omega \}\) be a sequence of \(k\)-covers of \(X\).
    It follows that \(\{ D(\mathscr U_n) : n \in \omega \}\) is a sequence of dense subsets of \(C_k(X)\) so we can apply \(S_1(\mathcal D_{C_k(X)}, \Omega_{C_k(X) , \mathbf 0})\) to obtain \(\{f_n : n \in \omega\} \in \Omega_{C_k(X) , \mathbf 0}\) so that \(f_n \in D(\mathscr U_n)\) for each \(n \in \omega\).
    For each \(n \in \omega\), let \(U_n \in \mathscr U_n\) be so that \(f_n[X \setminus U_n] \equiv 1\).

    We now show that \(\{U_n : n \in \omega\}\) is a \(k\)-cover of \(X\).
    Indeed, let \(K \in K(X)\) and let \(n \in \omega\) be so that
    \[
        f_n \in [ \mathbf 0; K , 1/2].
    \]
    Since \(f_n[ X\setminus U_n] \equiv 1\), \(K \cap (X\setminus U_n) = \emptyset\).
    Hence, \(K \subseteq U_n\), establishing that \(\{ U_n : n \in \omega \}\) is a \(k\)-cover of \(X\).

    \ref{OnePreWinD} \(\Rightarrow\) \ref{OnePreWinOmega}:
    This follows from \(\mathcal{D}_{C_k(X)} \subseteq \Omega_{C_k(X),\mathbf{0}}\).
    Thus \ref{OnePreWinkRoth}, \ref{OnePreWinOmega}, and \ref{OnePreWinD} are equivalent.

    \ref{OnePreWinOmega} \(\Rightarrow\) \ref{OneWinOmega}:
    Obvious

    \ref{OneWinOmega} \(\Rightarrow\) \ref{OnePreWinOmega}:
    We do this direction by the contrapositive.
    This is similar to the proof of Theorem 13 in \cite{Scheepers1997}.
    Suppose \(S_1(\Omega_{C_k(X),\mathbf{0}},\Omega_{C_k(X),\mathbf{0}})\) and let \(\sigma\) be a strategy for One in \(G_1(\Omega_{C_k(X),\mathbf{0}},\Omega_{C_k(X),\mathbf{0}})\).
    By \(S_1(\Omega_{C_k(X),\mathbf{0}},\Omega_{C_k(X),\mathbf{0}})\), any blade at \(\mathbf 0\) yields a countable subset which is also a blade at \(\mathbf 0\) so assume One is playing countable blades.
    If One ever plays a blade which contains \(\mathbf 0\), Two can choose \(\mathbf 0\) and thus win the game.
    So we assume that One only plays sets which don't contain \(\mathbf 0\).

    We will translate \(\sigma\) into a strategy \(\sigma^*\) for One in \(G_1(\mathcal K_X, \mathcal K_X)\).
    Suppose \(\sigma(\emptyset) = \{f_{\emptyset,n} : n \in \omega\}\).
    For each \(n\), let
    \[
        U_{\emptyset,n} = f_{\emptyset,n}^{-1}\left[\left(-\frac{1}{2},\frac{1}{2}\right)\right].
    \]
    Set \(\sigma^*(\emptyset) = \{U_{\emptyset,n} : n \in \omega\}\).
    We claim that this is a \(k\)-cover of \(X\).
    Indeed, let \(K \subseteq X\) be compact.
    Since the \(\{f_{\emptyset,n} : n \in \omega\}\) is a blade at \(\mathbf 0\), there exists \(N\) so that \(f_{\emptyset,N} \in [\mathbf 0;K,1/3]\).
    Hence, \(K \subseteq U_{\emptyset,N}\).

    Now Two will play some \(U_{\emptyset,n_0}\) in response, and we can translate this play back to \(C_k(X)\) by having Two play \(f_{\emptyset,n_0}\) in that game.
    Suppose \(\sigma(f_{\emptyset,n_0}) = \{f_{n_0,n} : n \in \omega\}\).
    Set
    \[
        U_{n_0,n} = f_{n_0,n}^{-1}\left[\left(-\frac{1}{2^2},\frac{1}{2^2}\right)\right].
    \]
    Set \(\sigma^*(U_{\emptyset,n_0}) = \{U_{n_0,n} : n \in \omega\}\).
    We can continue defining \(\sigma^*\) recursively in this way.

    We will show that \(\sigma\) is not a winning strategy for One.
    By \ref{OnePreWinkRoth} \(\Rightarrow\) \ref{OnePreWinOmega}, we have that
    \[
        S_1(\Omega_{C_k(X),\mathbf{0}},\Omega_{C_k(X),\mathbf{0}}) \implies S_1(\mathcal K_X,\mathcal K_X)
    \]
    which in turn implies that One does not have a winning strategy in \(G_1(\mathcal K_X, \mathcal K_X)\).
    Thus there is some play \(U_{\emptyset,n_0},U_{n_0,n_1},\cdots\) by Two which forms a \(k\)-cover.
    Consider the corresponding play \(f_{\emptyset,n_0}, f_{n_0,n_1}\) by Two against \(\sigma\).
    Set \(f_0 = f_{\emptyset,n_0}\), \(f_1 = f_{n_0,n_1}\) and so on, and likewise define \(U_n\).
    Let \(K \subseteq X\) be compact and \(\varepsilon > 0\).
    Then \(K \subseteq U_i\) for infinitely many \(i\).
    Thus
    \[
    K \subseteq f_{i}^{-1}\left[\left(-\frac{1}{2^{i+1}},\frac{1}{2^{i+1}}\right)\right]
    \]
    Therefore \(f_i \in [\mathbf 0; K,\varepsilon]\) for infinitely many \(i\).
    Since \(K\) and \(\varepsilon\) we arbitrary, this shows that \(\mathbf 0\) is in the closure of \(\{f_i : i \in \omega\}\).
    Thus Two has won this run of the game and \(\sigma\) is not a winning strategy.
    This shows that \ref{OneWinOmega} implies \ref{OnePreWinOmega}.

    \ref{OnePreWinD} \(\Rightarrow\) \ref{OneWinD}: Obvious.

    Proposition 21 of \cite{ClontzHolshouser} shows that \ref{OneWinD} \(\Rightarrow\) \ref{TwoWinCL} \(\Rightarrow\) \ref{TwoWinGru} \(\Rightarrow\) \ref{OneWinOmega}.

    Proposition 22 of \cite{ClontzHolshouser} shows that \ref{OnePreWinD} \(\Rightarrow\) \ref{TwoMarkCL} \(\Rightarrow\) \ref{TwoMarkGru} \(\Rightarrow\) \ref{OnePreWinOmega}.

    \ref{TwoMarkkKO} \(\Rightarrow\) \ref{TwoMarkCD}:
    Suppose \(\tau\) is a winning Markov strategy for Two in \(G_1(\mathscr N[K(X)], \neg\mathcal K_X)\).
    Let \([f;K,\varepsilon]\) be a subset of Player One's play in the \(n^{\text{th}}\) inning in \(G_1(\mathscr T_X , CD_{C_k(X)})\).
    Then let \(V = \tau(K,n)\).
    We can find a continuous \(g:X \to \mathbb R\) with the property that \(g|_K = f|_K\) and \(g[X \setminus V] = \{n\}\).
    We define \(\tau^*([f;K,\varepsilon],n) = g\).
    Then \(\tau^*\) is a Markov strategy for Two in \(G_1(\mathscr T_X , CD_{C_k(X)})\).

    Suppose \(U_n = [f_n;K_n,\varepsilon_n]\), \(n \in \omega\), is a play of \(G_1(\mathscr T_X , CD_{C_k(X)})\) by player One.
    We claim that \(\{g_n = \tau^*(U_n,n) : n \in \omega\}\) is closed discrete.
    Let \(f \in C_k(X)\).
    Set \(V_n = \tau(K_n,n)\) and notice that \(\{V_n : n \in \omega\}\) is not a \(k\)-cover.
    Thus we can find a compact \(K \subseteq X\) so that \(K \not\subseteq V_n\) for all \(n\).
    Hence, there is some \(x_n \in K \setminus V_n\) so that \(g_n(x_n) = n\).
    Since \(f\) is continuous, \(f[K]\) is bounded.
    Hence we can find an \(N\) so that \(f(x)+1 < N\) for all \(x \in K\).
    Then \(g_n \notin [f;K,1]\) for all \(n > N\).
    Thus \(f\) is not a limit point of \(\{g_n : n \in \omega\}\).
    Since \(f\) was arbitrary, \(\{g_n : n \in \omega\}\) is closed discrete.
    Therefore \(\tau^*\) is a winning mark for Two in \(G_1(\mathscr T_X , CD_{C_k(X)})\).

    Finally, \ref{TwoMarkCD} \(\Rightarrow\) \ref{TwoWinCD} \(\Rightarrow\) \ref{TwoWinCL}.
\end{proof}

\begin{cor}\label{PositiveSelectionPrinciples}
    For a Tychonoff space \(X\), the following are equivalent:
    \begin{enumerate}[label=(\alph*)]
        \item \label{thm:firstKSelection}
        \(S_1(\mathcal K_X , \mathcal K_X)\), that is \(\text{I} \underset{\text{pre}}{\not\uparrow} G_1(\mathcal K_X, \mathcal K_X)\)
        \item \label{thm:bladeCKSelection}
        \(S_1(\Omega_{C_k(X) , \mathbf 0} , \Omega_{C_k(X) , \mathbf 0})\), that is \(\text{I} \underset{\text{pre}}{\not\uparrow} G_1(\Omega_{C_k(X) , \mathbf 0} , \Omega_{C_k(X) , \mathbf 0})\)
        \item \label{thm:denseCKSelection}
        \(S_1(\mathcal D_{C_k(X)} , \Omega_{C_k(X) , \mathbf 0})\), that is \(\text{I} \underset{\text{pre}}{\not\uparrow} G_1(\mathcal D_{C_k(X)} , \Omega_{C_k(X) , \mathbf 0})\)
        \item \label{thm:CKsCFT}
        \(C_k(X)\) has strong countable fan tightness
        \item \label{thm:CKsCDFT}
        \(C_k(X)\) has strong countable dense fan tightness
        \item \label{thm:notMarkCD}
        \(\text{II} \underset{\text{mark}}{\not\uparrow} G_1(\mathscr T_{C_k(X)},{CD}_{C_k(X)})\)
        \item \label{thm:notMarkkKO}
        \(\text{II} \underset{\text{mark}}{\not\uparrow} G_1(\mathscr N[K(X)], \neg \mathcal K_X)\)
    \end{enumerate}
\end{cor}
\begin{proof}
    All of these equivalences follow from Theorem \ref{thm:FinalBoss} except for \ref{thm:bladeCKSelection} \(\Leftrightarrow\) \ref{thm:CKsCFT} and \ref{thm:denseCKSelection} \(\Leftrightarrow\) \ref{thm:CKsCDFT}.
    These follow from the fact that \(C_k(X)\) is a homogeneous space.
\end{proof}

\section{Further Questions}
\begin{itemize}
    \item
    How much of this theory can be recovered for \(C_k(X,[0,1])\)?
    \item
    To what extent can Propositions \ref{prop:Pawlikowski1} and \ref{prop:Pawlikowski2} be generalized?
    \item
    Does there exist a space \(X\) so that \(\text{I} \uparrow G_1(\mathscr N[K(X)], \neg\mathcal K_X)\) but One does not have a winning strategy in the finite-open game nor a pre-determined strategy in \(G_1(\mathscr N[K(X)], \neg\mathcal K_X)\)?
    \item
    How much of this theory carries over to longer length games?
\end{itemize}

\section{Acknowledgements}

The authors would like to thank Dr. Clontz for his valuable input during the writing of this manuscript and the organizers of the Spring Topology and Dynamical Systems 2018 Conference where the original inspiration for this work was born.

\providecommand{\bysame}{\leavevmode\hbox to3em{\hrulefill}\thinspace}
\providecommand{\MR}{\relax\ifhmode\unskip\space\fi MR }
\providecommand{\MRhref}[2]{%
  \href{http://www.ams.org/mathscinet-getitem?mr=#1}{#2}
}
\providecommand{\href}[2]{#2}

\end{document}